\documentclass[10pt]{amsart}

\usepackage{amsfonts}
\usepackage{amsthm}
\usepackage{amssymb}
\usepackage{graphicx}
\usepackage{subfigure}
\usepackage{xcolor}

\title[Constant Complexity Growth]{On the Number of Ergodic Measures for Minimal Shifts with Eventually Constant Complexity Growth}
\author[Damron]{Michael Damron}\email{mdamron6@gatech.edu}\thanks{The research of M. D. is supported by NSF grant DMS-0901534} 
\author[Fickenscher]{Jon Fickenscher}\email{jonfick@princeton.edu}
\date{\today}

\newcommand{\AAA}{\mathcal{A}}

\newcommand{\CCC}{\mathcal{C}}
\newcommand{\DDD}{\mathcal{D}}
\newcommand{\EEE}{\mathcal{E}}
\newcommand{\JJJ}{\mathcal{J}}
\newcommand{\KKK}{\mathcal{K}}
\newcommand{\LLL}{\mathcal{L}}
\newcommand{\MMM}{\mathcal{M}}

\newcommand{\SSS}{\mathcal{S}}

\newcommand{\WWW}{\mathcal{W}}

\newcommand{\CC}{\mathbb{C}}
\newcommand{\NN}{\mathbb{N}}

\newcommand{\RHScase}[1]{\left\{\begin{array}{ll} #1 \end{array}\right.}

\newcommand{\NIL}{\mathbf{0}}

\newcommand{\MEASURES}{\MMM}
\newcommand{\ERGODIC}{\EEE}
\newcommand{\SPACE}{\Omega}
\newcommand{\ERGODICsp}{\ERGODIC(\SPACE)}
\newcommand{\MEASURESsp}{\MEASURES(\SPACE)}
\newcommand{\LANGUAGE}{\LLL}
\newcommand{\LANGUAGEl}{\LANGUAGE_\ell}
\newcommand{\LANGUAGEr}{\LANGUAGE_r}
\newcommand{\LANGUAGEs}{\LANGUAGE_\mathfrak{s}}
\newcommand{\ALPHABET}{\AAA}
\newcommand{\RGRAPH}{\Gamma}
\newcommand{\COMPLEXITY}{p}
\newcommand{\SPRGRAPH}{\RGRAPH^{\mathrm{Spec}}}
\newcommand{\BASEGRAPH}{\Lambda}
\newcommand{\BASEGRAPHTOWER}{\Lambda^{\mathrm{Tower}}}
\newcommand{\BISPECIAL}{\mathrm{B}^{\mathrm{Sp}}}
\newcommand{\COLOR}{\CCC}
\newcommand{\PRECOLOR}{\mathfrak{C}}
\newcommand{\eps}{\varepsilon}
\newcommand{\LOOPTIMES}{\WWW}
\newcommand{\EMPTYWORD}{\epsilon}
\newcommand{\EXTEND}{\mathrm{Ext}}
\newcommand{\EXTENDr}{\EXTEND_r}
\newcommand{\EXTENDl}{\EXTEND_\ell}
\newcommand{\EXTENDs}{\EXTEND_\mathfrak{s}}
\newcommand{\EXTENDbs}{\EXTEND_{\ell r}}
\newcommand{\DENS}{\DDD}
\newcommand{\DENSL}{\DENS_L}
\newcommand{\DENSLN}{\DENS_{L,N}}
\newcommand{\DENSLMN}{\DENS_{L,MN}}

\newtheorem*{nnthm}{Theorem}

\newtheorem{lem}{Lemma}[section]
\newtheorem{cor}[lem]{Corollary}
\newtheorem{prop}[lem]{Proposition}
\newtheorem{thm}[lem]{Theorem}

\theoremstyle{plain}
\newtheorem{rmk}[lem]{Remark}
\theoremstyle{definition}
\newtheorem{defn}[lem]{Definition}
\newtheorem{assume}[lem]{Assumption}

\begin{document}

      \begin{abstract}
	    In 1985, Boshernitzan showed that a minimal (sub)shift
	    	satisfying a \emph{linear block growth}
	    	condition must have a bounded number of
	    	ergodic probability measures.
	    Recently, this bound was shown to be sharp
	    	through examples constructed by Cyr and Kra.
	    In this paper, we show that under the stronger assumption
	    	of \emph{eventually constant growth},
	    	an improved bound exists.
	    To this end, we introduce \emph{special Rauzy graphs}.
	    Variants of the well-known Rauzy graphs from symbolic dynamics,
		  these graphs provide an explicit description
		  of how a Rauzy graph for words of length $n$
		   relates to the one for words of length $n+1$
		    for each $n = 1,2,3\dots$.
      \end{abstract}

    \maketitle

    \section{Introduction}

    \subsection{Motivation and Main Result}

    For a finite alphabet $\ALPHABET$ of symbols, the set
	$$\ALPHABET^\NN = \{x = x_1x_2x_3\dots: x_n\in \ALPHABET\mbox{ for all }n\in\NN\}$$
    is endowed with the natural product topology and
	may be realized as a compact metric space.
    In this paper, $\NN = \{1,2,\dots\}$ is the set of positive integers and
	$\NN_0 = \{0\}\cup \NN$.
    The (left) shift $T: \ALPHABET^\NN \to \ALPHABET^\NN$ is defined by
	$$ (Tx)_n = x_{n+1} \mbox{ for all } n\in \NN$$
     and is continuous.
    A \emph{shift}\footnote{Many texts call $\ALPHABET^\NN$ a \emph{shift} and regard what we define here as a \emph{subshift}.
	We follow the convention of calling these objects \emph{full shifts} and \emph{shifts} respectively.}
      $\SPACE\subseteq \ALPHABET^\NN$ is any
      closed and $T$-invariant subset of $\ALPHABET^\NN$.
    We will restrict our discussion to \emph{minimal} shifts, meaning that every $T$-orbit is dense in $\SPACE$,
	or equivalently that there are no non-trivial shifts $\SPACE' \subsetneq \SPACE$.

    The set $\ALPHABET^* = \bigcup_{n\in \NN_0} \ALPHABET^n$ is the collection of all finite \emph{words} on $\ALPHABET$,
	including the empty word $\EMPTYWORD$.
    The \emph{language} of a shift $\SPACE$ is the collection of all words that occur in any $x\in \SPACE$, or
      $$ \LANGUAGE_{\SPACE} = \{w\in \ALPHABET^*: x_{[j,j+|w|-1]} = w \mbox{ for some } j\in \NN \mbox{ and }x\in \SPACE\}.$$
    Here $x_{[i,j]} = x_i x_{i+1} \dots x_{j-1} x_j$ represents the word in $x$ that begins at position $i$ and ends
	  at position $j$ and $|w|$ is the \emph{length} of $w$; that is, $|w|=n$ where $w = w_1w_2\dots w_n$.
    We may then define $\LANGUAGE_{\SPACE}(n)$ for $n\in \NN_0$ as the set of all
	$w\in \LANGUAGE_{\SPACE}$ such that $|w| = n$.

    One object that has been used to describe shifts is the \emph{complexity function} $\COMPLEXITY_{\SPACE}$,
	defined as
	$$ \COMPLEXITY_{\SPACE}(n) = \# \LANGUAGE_{\SPACE}(n).$$
    For example, the Morse-Hedlund Theorem shows that any minimal $\SPACE$ whose complexity function satisfies
	$\COMPLEXITY_{\SPACE}(n_0) \leq n_0$ for some $n_0$ must actually have bounded complexity for all $n$
	and must therefore be a finite and periodic system.
    As a result, $\COMPLEXITY_{\SPACE}(n) \geq n+1$ for all $n$ if $\SPACE$ is aperiodic, 
      and the class of well-studied $\SPACE$ such that equality holds for all $n$ is
      known as \emph{Sturmian}.

    When considering the Borel $\sigma$-algebra for minimal $\SPACE$,
	  the system may be viewed as a measure-theoretic dynamical system.
	Boshernitzan in \cite{cBosh1985} wanted to describe the set of
	$T$-invariant probability measures $\MEASURESsp$ by
	 bounding the size of the set of ergodic measures
	 $\ERGODICsp\subseteq \MEASURESsp$ for $\SPACE$'s such that
	 $\COMPLEXITY_{\SPACE}$ satisfies some linear upper bounds.
    In particular, he showed the following results.

    \begin{nnthm}[Boshernitzan]
	  Let $\SPACE$ be an aperiodic minimal shift on a finite alphabet $\ALPHABET$.
	    \begin{enumerate}\renewcommand{\theenumi}{\roman{enumi}}
	     \item  \cite[Corollary 1.3]{cBosh1985}

		  If $\displaystyle \liminf_{n\to\infty} \frac{\COMPLEXITY_{\SPACE}(n)}{n} = \alpha$, then $\#\ERGODICsp \leq \lfloor \alpha\rfloor$.

	      \item\label{ThmBosh2} \cite[Theorem 1.5 \& Theorem 8.1]{cBosh1985}

		  If $\displaystyle \limsup_{n\to\infty} \frac{\COMPLEXITY_{\SPACE}(n)}{n} <K$ for some integer $K\geq 3$, then $\#\ERGODICsp \leq K-2$.
	      \item\label{ThmBosh3} If $\displaystyle \limsup_{n\to\infty} \frac{\COMPLEXITY_{\SPACE}(n)}{n} = \alpha$ for real $\alpha\geq 2$,
		      then $\#\ERGODICsp \leq \lfloor \alpha\rfloor-1$.
	    \end{enumerate}
    \end{nnthm}
    Note that (\ref{ThmBosh3}) is implied by (\ref{ThmBosh2}) by choosing $K = \lfloor \alpha \rfloor +1$.

    For any integer $d\geq 3$, V. Cyr and B. Kra \cite{cCyrKra2015} recently constructed minimal shifts $\SPACE$ such that
      $$ \liminf_{n\to \infty} \frac{\COMPLEXITY_{\SPACE}(n)}{n} = d,~
	    \limsup_{n\to \infty} \frac{\COMPLEXITY_{\SPACE}(n)}{n} = d+1
	    \mbox{ and } \#\ERGODICsp=d.$$
      These demonstrate that Boshernitzan's results are sharp.
    They also strengthened the results in Bohsernitzan's paper by
	allowing non-minimal $\SPACE$ and achieving the same bound
	for the larger class of \emph{generic measures}.
    A measure $\mu$ on $\SPACE$ is generic if there exists $x\in \SPACE$ so that
	      $$ \lim_{N\to\infty} \frac{1}{N} \sum_{n=0}^{N-1} f(T^nx) = \int_{\SPACE} f d\mu$$
    for all continuous $f:\SPACE \to \CC$. Note that any such $f$ is bounded due to compactness of $\Omega$.

    In this paper, we improve Boshernitzan's results under stronger assumptions on $\SPACE$.
    We are motivated by the class of shifts associated to \emph{interval exchange transformations}.
	See \cite{cVia2006} for a survey of these dynamical systems and
	  \cite{cFerenZam2008} regarding their associated shifts.
    The following facts hold in generality\footnote{Meaning for all interval exchange transformations that
		    satisfy the \emph{infinite distinct orbit condition}, a generic condition introduced in \cite{cKea1975}.}:
	    if $\SPACE$ is a shift associated to a minimal interval exchange on $d$ intervals then
	  $$ \COMPLEXITY_{\SPACE}(n) = (d-1)n + 1 \mbox{ for all }n$$
    while $\#\ERGODICsp\leq \lfloor d/2\rfloor$, as proved by \cite{cKatok1973} and later,
    	with a different method,
	by \cite{cVe1990}.
    The $\lfloor d/2\rfloor$ bound was verified to be sharp on $4$ intervals in \cite{cKea1977}
	and then for all $d\geq 4$ in \cite{cYoc2007}.
    For $d\leq 4$, this bound and Boshernitzan's bound agree.
    However, for $d\geq 5$ Boshernitzan's bound, $\#\ERGODICsp \leq d-2$, is strictly weaker.

    We will consider minimal $\SPACE$ whose complexity function satisfies an
      \emph{eventually constant growth condition}: $\COMPLEXITY_{\SPACE}(n+1) - \COMPLEXITY_{\SPACE}(n) = K$ for a fixed $K\in \NN$ and
      all $n\geq n_0$ for some $n_0$.
    Equivalently, $\COMPLEXITY_{\SPACE}(n)$ has eventually constant growth if and only if
	\begin{equation}\label{EqFixedKComp}
	      \COMPLEXITY_{\SPACE}(n) = Kn + C \mbox{ for all }n\geq n_0,
	\end{equation}
    where $K,n_0\in \NN$ and $C \in \NN_0$ are constants.

    We now present our main result.
	  \begin{thm}\label{ThmMain}
		If a minimal shift $\SPACE$ on a finite $\ALPHABET$ satisfies equation \eqref{EqFixedKComp} with $K \geq 4$,
		then $\#\ERGODICsp \leq K-2$.
	  \end{thm}

Note that for such a space $\SPACE$, Boshernitzan's result gives $\#\ERGODICsp \leq K-1$. This is also the bound given by Kra and Cyr \cite{cCyrKra2015}, though their results apply to more general systems. So our bound of $K-2$ is a strict improvement over the previous ones for ergodic measures, in our setting of eventually constant growth.

    \subsection{Outline of paper}

	In Section \ref{SecDef}, we establish the notations and definitions used in this paper.
	While most of the ideas presented are well-known, we do introduce
		two concepts vital to our work.
	In Section \ref{SsecSpRG} we define \emph{special Rauzy graphs},
		variants on Rauzy graphs from symbolic dynamics.
	We then define the \emph{binary extension condition} for a
		language/shift in Section \ref{Ssec2EC}.
		
	We define and prove results for a notion of \emph{disjoint density},
		motivated by ideas from \cite{cBosh1985},
		in Section \ref{SecDens}.
	Loosely speaking, a measure $\mu$ has \emph{disjoint density}
		$\beta>0$ in a measure $\nu$ if a fixed sequence of
		words generating $\mu$ occurs with a frequency at least $\beta$
		in a generic sequence $x$ for $\nu$.
	We primarily use disjoint density for ergodic $\nu$,
		and in this case if $\mu$ has positive disjoint density in $\nu$,
		then $\mu = \nu$ (see Corollary \ref{CorDensMeas}).
		
	A \emph{coloring function} on special Rauzy graphs
		is introduced in Section \ref{SecColor}.
	We show that any such coloring function must satisfy
		a set of rules (Proposition \ref{PropColorCycles})
		and the number of colors for such a function
		bounds the number of ergodic measures (Definition \ref{DefColor}).
		
	There is a special Rauzy graph for each $n\in \NN$ and
		the graph for $n$ is related to the graph for $n+1$
		by \emph{bispecial moves}.
	Defined in Section \ref{SecBiSpMoves},
		such moves explicitly describe all possible changes
		as $n$ increases.
	We describe the effects of such moves on coloring functions
		in Lemma~\ref{LemColorMoves} for different graphs
		and end this section by considering \emph{loops}
		in special Rauzy graphs.
	These are pairs of vertices that form a cycle in the
		graph and represent the smallest
		set of vertices that may share a color.
	In our main proofs, we look at such loops to force
		measures (i.e., colors) to ``spread''
		in graphs with too many loops.
	To achieve this, we establish necessary results in Section \ref{SsecColoringLoops}.
	
	The proof of the main theorem is provided in Section \ref{SecMainProof}.
	We first show our result for very specific graphs
		(Lemmas \ref{LemMainOneTower} and \ref{LemMainManyLoops}).
	These graphs are composed of many consecutive loops, allowing
		for freedom in only a few vertices.
	We then provide a proof of our main theorem under the binary extension condition
		as defined in Section \ref{Ssec2EC}.
	The section ends with a proof for all shifts that satisfy equation \eqref{EqFixedKComp}.
	
	We end with Section \ref{SecFuture} by listing further uses for the tools developed in this work.
	In particular, if we make stronger assumptions on our shift then
		we may achieve a better bound for $\#\ERGODICsp$ than in
		Theorem \ref{ThmMain}.
	
    \section{Definitions}\label{SecDef}

	When considering a minimal shift $\SPACE$ on finite alphabet $\ALPHABET$,
	    we will typically suppress the subscript $\SPACE$ when referring to
	    the complexity function $\COMPLEXITY = \COMPLEXITY_{\SPACE}$ and language $\LANGUAGE = \LANGUAGE_{\SPACE}$.

      \subsection{Ergodic Theory}

	  The topology on $\ALPHABET^\NN$, and therefore on any shift $\SPACE\subseteq \ALPHABET^\NN$,
	      is generated by \emph{cylinder sets} $[w]$ for words $w\in \ALPHABET^*$, where
	    $$ [w] = \{x\in \ALPHABET^\NN: x_{[1,n]} = w\},~n = |w|.$$
	  In other words, $[w]$ is the collection of all $x$ such that $x_1\dots x_n = w$.
	  Cylinders are \emph{clopen}; that is, closed and open,
		and the indicator functions $\chi_{[w]}$
		form a countable basis for $\CCC(\AAA^\NN)$,
		the space of continuous functions $\AAA^\NN \to \CC$.
	  The metric
		  $$ d(x,y) = \RHScase{
				  2^{-\min\{n\in \NN: x_n \neq y_n\}}, & x\neq y,\\
				  0, & x = y,
				  }$$
	  also generates the same topology,
	    and it follows that any shift is a compact metric space.
	  Any measure $\mu\in \MEASURESsp$ on a shift $\SPACE\subseteq\ALPHABET^\NN$
	       naturally extends to $\ALPHABET^\NN$ by
	      defining $\mu(\ALPHABET^\NN \setminus \SPACE) = 0$.
	  By the Riesz Representation Theorem,
	  	$\mu\in \MEASURESsp$ is then uniquely determined by the values
		$$\mu([u]) \mbox{ for } u\in \ALPHABET^\NN$$
	  and for minimal $\SPACE$, $\mu([u]) >0$ if and only if $u\in \LANGUAGE$.
	  The reader should refer to Sections 6.1--6.2 in \cite{cWalters2000} for more background on
		invariant measures for compact metric spaces.
	  If $\mu,\mu'\in \MEASURESsp$ and $\beta\in (0,1)$, we say that
	      $$ \mu \geq \beta \mu' $$
	  when $\mu([u]) \geq \beta \mu'([u])$ for all $u\in \ALPHABET^*$.
	  \begin{rmk}\label{RmkErgExtreme}
	  By the extremality of $\ERGODICsp$ in $\MEASURESsp$, for any $\nu\in \ERGODICsp$,
		      $\mu\in \MEASURESsp$ and $\beta\in(0,1)$,
		$ \nu \geq \beta \mu$ implies $\nu = \mu.$
	   \end{rmk}

	  For $u,w\in \ALPHABET^*$, let
	      $$ \#_u(w) = \#\{1\leq j \leq |w|-|u|+1: w_{[j,j+|u|-1]} = u\}$$
	  denote the number of occurrences of $u$ in $w$.
	  If $\nu \in \ERGODICsp$, then $\nu$-almost every $x\in \SPACE$ is
	      \emph{generic} for $\nu$ by Birkhoff's Pointwise Ergodic Theorem, meaning
	      $$ \lim_{n\to \infty} \frac{\#_u(x_{[1,n]})}{n} = \nu([u])\mbox{ for every } u\in \ALPHABET^*.$$
	  \begin{defn}\label{DefFixedGeneric}
		For each $\nu\in \ERGODICsp$, we fix $x^{(\nu)}\in \SPACE$ that is generic for $\nu$.
	  \end{defn}

	  Let $(w^{(n)})_{n\in \NN}$ be a sequence of words $w^{(n)}\in \ALPHABET^*$ such that $|w^{(n)}| \to \infty$
		  as $n\to\infty$.
	  If for each $u\in \ALPHABET^*$ the limit
		    \begin{equation}\label{EqWordsToMeas}
		    		\phi(u) = \lim_{n\to\infty} \frac{\#_u(w^{(n)})}{|w^{(n)}|}
		    \end{equation}
	  exists, then there is a unique $T$-invariant measure $\mu\in \MEASURES(\ALPHABET^\NN)$
		 such that $\mu([u]) = \phi(u)$ for all $u\in \ALPHABET^*$.
	  Furthermore, if $w^{(n)} \in \LANGUAGE$ for all $n$, then $\mu\in \MEASURESsp$.

	  \begin{defn}\label{DefWordsToMeas}
		If equation \eqref{EqWordsToMeas} holds for all $u\in \ALPHABET^*$ for a sequence of words $(w^{(n)})_{n\in \NN}$
		  where $w^{(n)} \in \LANGUAGE$ for all $n$ as above and $\mu$ is the associated measure, we say that $(w^{(n)})_{n\in \NN}$
		    \emph{generates} $\mu$ or $w^{(n)} \to \mu$ as $n\to\infty$.
	  \end{defn}

	  \begin{rmk}\label{RmkWordsToMeas}
		Given a sequence $(w^{(n)})_{n\in \NN}$ in $\LANGUAGE$ such that $|w^{(n)}|\to \infty$,
		       the limit in equation \eqref{EqWordsToMeas} might not exist for all $u \in \LANGUAGE$.
		However, by diagonalization, we may choose a subsequence $(n_k)_{k\in \NN}$ so that
			  $$ \lim_{k\to\infty} \frac{\#_u(w^{(n_k)})}{|w^{(n_k)}|}$$
		exists for all $u\in \ALPHABET^*$.
		In this case, we still obtain $\mu\in \MEASURESsp$ and
		  write $w^{(n)} \to \mu$ for $\JJJ\ni n \to \infty$ where $\JJJ = \{n_k: k\in \NN\}$.
	  \end{rmk}

      \subsection{Special Words}

      For a minimal $\SPACE$, the language $\LANGUAGE$ has the following properties:
	    \begin{itemize}
	      \item $\LANGUAGE$ contains all of its subwords, meaning that if $w\in \LANGUAGE$ then $w_{[i,j]}\in \LANGUAGE$
			for any $i,j$ with $1\leq i \leq j \leq |w|$.
	     \item Every word $w\in\LANGUAGE$ is \emph{extendable}, meaning there exist $a,b\in \ALPHABET$ so that
		      the concatenation $awb$ is an element of $\LANGUAGE$.
	    \end{itemize}
      For any $w\in \LANGUAGE(n)$, we define the \emph{left extensions} and \emph{right extensions} respectively by
		$$ \EXTENDl(w) = \{w'\in \LANGUAGE(n+1): w'_{[2,n+1]} = w\}$$
		and
		$$ \EXTENDr(w) = \{w'\in \LANGUAGE(n+1): w'_{[1,n]} = w\}.$$
      Likewise,
	  let
	      $$ \EXTENDbs(w) = \{w''\in \LANGUAGE(n+2): w''_{[2,n+1]} = w\}$$
	  denote the \emph{two-sided extensions} of $w$.
	Because $\LANGUAGE$ is extendable, these sets are all non-empty for each $w$.
	A word $w\in \LANGUAGE$ is \emph{left special} if $\#\EXTENDl(w) \geq 2$ 
	    and $w$ is \emph{right special} if $\EXTENDr(w) \geq 2$.
	A \emph{bispecial word} is one that is both left and right special.
	Let $\LANGUAGEl$ and $\LANGUAGEr$ denote the
	    left special and right special words in $\LANGUAGE$ respectively.
	For convenience, we will sometimes call $w$ $\mathfrak{s}$-special for
		  $\mathfrak{s}\in \{\ell,r\}$ to
		indicate that $w\in \LANGUAGEs$.

	For any $n\in \NN$ and $\mathfrak{s}\in \{\ell,r\}$,
		the sets $\EXTENDs(w)$, $w\in \LANGUAGE(n)$,
		partition $\LANGUAGE(n+1)$.
	Also, $\#\EXTENDs(w) = 1$ if and only if $w$ is not $\mathfrak{s}$-special.
	Therefore
		  \begin{equation}\label{EqExtToGrowth}
			\sum_{w\in \LANGUAGEs(n)} (\#\EXTENDs(w) - 1)
			  = \sum_{w\in \LANGUAGE(n)} (\#\EXTENDs(w) - 1) = \COMPLEXITY(n+1)-\COMPLEXITY(n),
		  \end{equation}
	where $\LANGUAGEs(n) = \LANGUAGEs \cap \LANGUAGE(n)$.
	We therefore have the following relationships between special words and growth of the
	    complexity function for aperiodic $\SPACE$. First, by using equation \eqref{EqExtToGrowth} and the fact that $\#\EXTENDs(w) \geq 2$ for all $w \in \LANGUAGEs(n)$,
		$$
		      1\leq \#\LANGUAGEs(n) \leq \COMPLEXITY(n+1)-\COMPLEXITY(n)
		      	\mbox{ for all }
		      n\in \NN, \mathfrak{s}\in \{\ell,r\}.
		$$
Furthermore, 
		\begin{equation}\label{Eq2ECComp}
		      \#\LANGUAGEs(n)= \COMPLEXITY(n+1)-\COMPLEXITY(n)
		\end{equation}
	holds for some $n\in \NN$ and $\mathfrak{s} \in \{\ell,r\}$ if and only if $\#\EXTENDs(w) = 2$ for all
		  $w\in \LANGUAGEs(n)$, or equivalently
		  \begin{equation}\label{Eq2EC}
			\max\{\#\EXTENDs(w): w\in \LANGUAGE(n)\} \leq 2.
		  \end{equation}

	\begin{lem}\label{LemDecPsi}
	      For $\mathfrak{s}\in \{\ell,r\}$, let $\psi_\mathfrak{s}:\NN_0 \to \NN$ be
		    defined as
			$$ \psi_\mathfrak{s}(n) = \max\{\#\EXTENDs(w): w\in \LANGUAGE(n)\}.$$
	      The function $\psi_{\mathfrak{s}}$ is non-increasing in $n$ and therefore is eventually
		    constant.
	\end{lem}

	\begin{proof}
	      We provide the proof when $\mathfrak{s} = \ell$, as the $\mathfrak{s} = r$ case is similar.
	      Consider any $w\in \LANGUAGE(n+1)$ and its length $n$ prefix $\tilde{w} = w_{[1,n]}$. 
	      We claim that the map
		      $$ w' \mapsto w'_{[1,n+1]}$$
	      is a well-defined injection from $\EXTENDl(w)$ to $\EXTENDl(\tilde{w})$.

	      For any $w'\in \EXTENDl(w)$, $w'_{[2,n+1]} = \tilde{w}$,
			so the image is a left-extension of $\tilde{w}$.
	      Each word in $\EXTENDl(w)$ is distinguished uniquely by its first letter.
	      It follows that $w'_{[1,n+1]} \neq w''_{[1,n+1]}$ for distinct $w',w''\in \EXTENDl(w)$,
		    proving injectivity.
	      Therefore $\psi_\ell(n+1) = \max_{w \in \LANGUAGE(n+1)} \#\EXTENDl(w) \leq \psi_\ell(n)$.
	\end{proof}

	We end this section by relating $\psi_\mathfrak{s}(n)$ as defined above to $\COMPLEXITY(n+1) - \COMPLEXITY(n)$
		    and $\#\LANGUAGEs(n)$.
	For each $w\in \LANGUAGEs(n)$, $\#\EXTENDs(w)\geq 2$ by
		  definition.
	Also, there exists $\tilde{w}\in \LANGUAGEs(n)$ such
		  that $\#\EXTENDs(\tilde{w}) = \psi_\mathfrak{s}(n)$.
	Applying this to equation \eqref{EqExtToGrowth},
		we obtain $(\#\LANGUAGEs(n) - 1) + (\psi_\mathfrak{s}(n)-1) \leq \COMPLEXITY(n+1) - \COMPLEXITY(n)$, or
	  \begin{equation}\label{EqPsiSpecBd}
		  \#\LANGUAGEs(n) + \psi_\mathfrak{s}(n) \leq \COMPLEXITY(n+1)-\COMPLEXITY(n)+2,
	  \end{equation}
	by bounding $\#\EXTENDs(w)$ from below by $2$ for
		all $\mathfrak{s}$-special $w\neq \tilde{w}$.

      \subsection{Special Rauzy Graphs}\label{SsecSpRG}
	  We first recall the definition of the \emph{Rauzy graphs} $\RGRAPH(n)$ for $n\in \NN$ associated to
		a language $\LANGUAGE$.
	  Each $\RGRAPH(n)$ is a directed graph with vertex set $\LANGUAGE(n)$ and a directed edge from $u$ to $v$,
	      written $u\to v$, if and only if there exists $w\in \LANGUAGE(n+1)$ such that $w_{[1,n]}=u$ and $w_{[2,n+1]} = v$.

	  We now define the \emph{special Rauzy graphs} $\SPRGRAPH(n)$.
	   If $w\in \LANGUAGE(n)$ is \emph{unispecial} (that is, left special or right special but not both),
	      then $w$ is a vertex in $\SPRGRAPH(n)$.
	   If $w\in \LANGUAGE(n)$ is bispecial, then we associate to it two distinct vertices $w_\ell$ and $w_r$ in $\SPRGRAPH(n)$.
	   An edge from unispecial $w$ to unispecial $w'$ exists, written $w\to w'$, when
		there is a path in $\RGRAPH(n)$ from $w$ to $w'$ that visits only non-special words in between.
	    All paths that end at a bispecial word $w$ in $\RGRAPH(n)$ will have their corresponding edges in $\SPRGRAPH(n)$ end at $w_\ell$,
		while all paths that begin at $w$ in $\RGRAPH(n)$ will have their corresponding edges in $\SPRGRAPH(n)$ begin at $w_r$.
	    We also include the edge $w_\ell\to w_r$ for each bispecial word $w$.
	    The \emph{weight} of edge $w\to w'$ in $\SPRGRAPH(n)$, denoted $\rho_n(w\to w')$, is the length of the corresponding path
		in $\RGRAPH(n)$, with the convention that $\rho_n(w_\ell \to w_r)=0$ for any bispecial $w\in \LANGUAGE(n)$.

	    \begin{defn}
		Given a special Rauzy graph $\SPRGRAPH(n)$, $w\in \SPRGRAPH(n)$ denotes that $w$ is a vertex in the graph
		      while $w\to w'\in \SPRGRAPH(n)$ means that the directed edge from vertex $w$ to vertex $w'$ exists in
		      the graph.
	    \end{defn}

	    We inherit the definitions from the language and refer to a vertex in $\SPRGRAPH(n)$ with more than one outgoing edge
		      as \emph{right special} and
	      a vertex with more than one incoming edge as \emph{left special}.
	    Note that every vertex in $\SPRGRAPH(n)$ is either left or right special but not both.
	    We call an edge $u\to v$ such that $u$ is left special and $v$ is right special a
		  \emph{bispecial edge}.

	    \begin{defn}
		  Given a special Rauzy graph $\SPRGRAPH(n)$ and $\mathfrak{s}\in\{\ell,r\}$ we let
			  $$ K_\mathfrak{s}(n) = \#\LANGUAGEs(n)$$
		   denote the number of $\mathfrak{s}$-special vertices in the graph.
	    \end{defn}

	\subsection{Binary Extension Condition}\label{Ssec2EC}
	The condition used in this paper on $\LANGUAGE$ may now be defined.
	 \begin{defn}\label{Def2EC}
	      A language $\LANGUAGE$ satisfies the \emph{binary $\mathfrak{s}$-extension condition} for $N_0$,
	      		$\mathfrak{s}\in \{\ell,r\}$,
		    if equation \eqref{Eq2EC} holds for all $n\geq N_0$.
	      If $\LANGUAGE$ satisfies both the binary $\ell$-extension condition and the binary
	      	$r$-extension condition for $N_0$,
		    then $\LANGUAGE$ satisfies the \emph{binary extension condition} for $N_0$.
	  \end{defn}
	  \begin{rmk}\label{Rmk2LettersMeans2EC}
		If $\ALPHABET = \{0,1\}$, then $\#\EXTENDl(w) \leq 2$ and $\#\EXTENDr(w)\leq 2$ for all $w$.
		Therefore any $\SPACE$ on $\ALPHABET = \{0,1\}$ will have a language that satisfies
		      the binary extension condition for all $n$.
		The results in this paper that follow will usually assume $\ALPHABET=\{0,1\}$ for convenience
		      but may be extended to any language with the binary extension condition after
		      ignoring finitely many $n$.
	  \end{rmk}

	  If $\SPACE$ on $\ALPHABET$ has language $\LANGUAGE$ that satisfies the binary extension condition for $N_0$
		  and has constant complexity growth $K$ as in equation \eqref{EqFixedKComp} for $n\geq N_0$ as well,
		  then for each special Rauzy graph $\SPRGRAPH(n)$ where $n\geq N_0$,
		$$
			K_\ell(n) = K_r(n) = K,
		$$
	  or each special graph has exactly $2K$ vertices.

	  The following natural consequence of Lemma \ref{LemDecPsi} will help to classify different
		  languages in the proof of Theorem \ref{ThmMain}.
	  Essentially, a language will either satisfy the binary extension condition or
		  will always have at least one $\mathfrak{s}$-special vertex with more than
		  two branches in each special Rauzy graph $\SPRGRAPH(n)$.

	  \begin{cor}\label{Cor2ECPropogates}
		Let $\SPACE$ be a minimal shift on finite $\ALPHABET$ with language $\LANGUAGE$.
		If \eqref{Eq2ECComp} holds for some $N_0$ and $\mathfrak{s}\in \{\ell,r\}$,
		 then $\LANGUAGE$ satisfies the binary $\mathfrak{s}$-extension condition for $N_0$.
	  \end{cor}

	  \begin{proof}
		  If equation \eqref{Eq2ECComp} holds for $N_0$,
			then $\psi_\mathfrak{s}(N_0)=2$,
		  where $\psi_\mathfrak{s}$ is from Lemma \ref{LemDecPsi}.
		  By that lemma, $\psi_\mathfrak{s}(n) \leq 2$ for all $n\geq N_0$.
	  \end{proof}

	  If $\LANGUAGE$ satisfies the binary extension condition, then for all large $n$
		    $$ \#\EXTENDbs(w) \in \{2,3,4\} $$
	  for a bispecial word $w$.
	  We may classify $w$ according to the number of two-way extensions,
	      using the terminology from \cite{cCass1996}.
	    \begin{itemize}
	     \item If $\#\EXTENDbs(w) = 2$, then $w$ is \emph{weak bispecial}.
			In this case, an extension of $w$ on one side uniquely determines the extension on the other.
	      \item If $\#\EXTENDbs(w) = 3$, then $w$ is \emph{regular bispecial}.
			Here exactly one right extension is left special and exactly one left extension is
			    right special.
	      \item If $\#\EXTENDbs(w) = 4$, then $w$ is \emph{strong bispecial}.
			  All one-sided extensions of $w$ are special on the opposite side.
	    \end{itemize}
	  Unless we assume the binary extension condition, bispecial words may not be as easily classified because
		the possible number of two way extensions for a given bispecial word may take on many more values.

    \section{Disjoint Density}\label{SecDens}

    \subsection{Definition}

    Let $w$ be a word on $\ALPHABET$ of length $n$, $L\in \NN$ and $x \in \ALPHABET^\NN$.
    We define $r_{L}(w,x,j)$ for $j\in \NN$ by
	$$
	      r_{L}(w,x,j) = \RHScase{1, & x_{[i,i+n-1]}= w \mbox{ for some }(j-1)Ln< i \leq jLn,\\
				    0, & \mbox{otherwise}.}
	$$
    So $r_L(w,x,j)$ indicates whether or not $w$ \emph{begins} anywhere in the $j^{th}$ block of length $Ln$ in $x$.
      We then define the \emph{sum function} and \emph{average function} as
	  $$
	      \SSS_{L,N}(w,x) = \sum_{j=1}^N r_L(w,x,j) \mbox{ and } \DENSLN(w,x) = \frac{1}{N}\SSS_{L,N}(w,x)
	  $$
      respectively.
    \begin{defn}
	  The \emph{disjoint (upper) density} of $w$ in $x$ by $L$-blocks is
	    $$
		\DENSL(w,x) = \limsup_{N\to \infty} \DENSLN(w,x).
	    $$
    \end{defn}

    \begin{rmk}\label{RmkDensSubseq}
	It is a direct exercise to show that
	      $$
				\DENSL(w,x) = \limsup_{N\to\infty} \DENSLMN(w,x)
	      $$
	for any $M\in \NN$.
    \end{rmk}

    \begin{rmk}
	This concept of density is similar in spirit to that in \cite{cBosh1985}.
	However, we are counting occurrences of $w$ that \emph{begin in} one $Ln$-block, including
	    $w$'s that end in the next block,
	    while the
	    analogous count in Boshernitzan's paper only allows for $w$ that are \emph{contained in} an $Ln$-block.
	This difference will be needed to prove Lemma \ref{LemNbringDens}, but
	    may be regarded as technical on first reading.
    \end{rmk}

    \subsection{Disjoint Density of Measures}\label{SsecDensMeas}

    Consider an infinite $\JJJ\subseteq \NN$ and corresponding sequence of words
	    $(w^{(n)})_{n\in\JJJ}$ where $|w^{(n)}|=n$ for each $n\in \JJJ$.
    Suppose $w^{(n)} \to \mu$ as $\JJJ\ni n \to \infty$ in the sense of
	Definition \ref{DefWordsToMeas} and Remark \ref{RmkWordsToMeas}.
     For $\nu\in \ERGODICsp$, let $x = x^{(\nu)}$ be the fixed generic point for $\nu$ from Definition \ref{DefFixedGeneric}.

    \begin{defn}
	For $\JJJ$, $(w^{(n)})_{n\in \JJJ}$, $\mu$, $x$ and $\nu$ above, the
	      \emph{disjoint (upper) $L$-density} of $\mu$ in $\nu$ is
	  $$
		\DENSL(\mu,\nu) = \limsup_{\JJJ\ni n\to \infty} \DENSL(w^{(n)},x).
	  $$
    \end{defn}

    Up to our change in definition from the original work,
	the proof of the next lemma is the same as for \cite[Lemma 4.5]{cBosh1985}.

    \begin{lem}\label{LemDensMeas}
	If for the notations in this section, $\beta := \DENSL(\mu,\nu) >0$ for some $L\in \NN$ then
		$$
		      \nu \geq \frac{\beta}{2L} \mu.
		$$
    \end{lem}

    \begin{proof}
	It suffices to show that 
	    $$
		\nu([u]) \geq \frac{\beta}{2L} \mu([u])
	    $$
	for an arbitrary fixed $u\in \ALPHABET^*$ such that $\mu([u])>0$.
	Fix $\eps>0$ such that $\eps< \min\{\mu([u]),\beta/2\}$, and choose $M_0$ so that
	      $$
		  \nu([u]) \geq \frac{1}{M} \#_u(x_{[1,M]}) -\eps
	      $$
	for all $M\geq M_0$.
	Choose $n\in \JJJ$ so that
	      $$
		  \frac{\#_u(w^{(n)})}{n} \geq \mu([u]) - \eps.
	      $$
	Finally, choose $N$ so that $NLn > M_0$, $N\eps\geq 1$ and
	      $$
		  \DENSLN(w^{(n)},x) = \frac{1}{N}\SSS_{L,N}(w^{(n)},x) \geq \beta - \eps.
	      $$
	It is possible that $u$ may occur in an overlap of at most two occurrences of $w^{(n)}$ beginning in adjacent
	    $nL$ blocks in $x$.
	Therefore by excluding the possible occurrence of $w^{(n)}$ in the final $nL$ block,
	      $$
		    \#_u(x_{[1,NLn]}) \geq \frac{1}{2}(\SSS_{L,N}(w^{(n)},x)-1) \cdot \#_u(w^{(n)}).
	      $$
	 It follows that
	      $$
		  \begin{array}{rcl}
			\nu([u]) & \geq & \frac{1}{NLn} \#_u(x_{[1,NLn]}) -\eps\\
				  & \geq & \frac{\#_u(w^{(n)})}{2NLn} (\SSS_{L,N}(w^{(n)},x)-1)- \eps\\
				  & \geq & \frac{1}{2L}\frac{\#_u(w^{(n)})}{n}\left(\frac{S_{L,N}(w^{(n)},x)}{N} - \frac{1}{N}\right) - \eps\\
				  & \geq & \frac{1}{2L}\left(\mu([u]) - \eps\right)\left(\beta - 2\eps\right) - \eps.
		  \end{array}
	      $$
	  By letting $\eps \to 0$, we arrive at the desired inequality.
    \end{proof}

    The following is a direct consequence of the previous lemma and Remark \ref{RmkErgExtreme}.

    \begin{cor}\label{CorDensMeas}
	If the conditions of Lemma \ref{LemDensMeas} hold and $\nu$ is ergodic, then $\nu = \mu$.
    \end{cor}

    \subsection{Relationships Between Densities}

    In this section, we derive some counting tools to work with densities. The main one is Lemma~\ref{LemNbringDens}, which implies that if a word $u$ appears with a positive frequency in a sequence $x$, and each occurrence of $u$ is associated with an occurrence of a word $w$ (and the distance between $u$ and $w$ is not too large), then $w$ also occurs in $x$ with a positive frequency. We note that the simpler result in Lemma \ref{LemSubwdDens} is a special case of Lemma \ref{LemNbringDens} and can be replaced
	  without ultimately affecting any results in this work.
    However, Lemma \ref{LemSubwdDens} has a better lower bound
	  when it applies.
    Furthermore, the statement and proof of Lemma \ref{LemSubwdDens} are both easier to read
	  and so we include the result to aid in the
	  understanding of the more technical result that follows.

    \begin{lem}\label{LemSubwdDens}
	  If $u$ is a subword of $w$, then
		$$
		      \DENSL(u,x) \geq \frac{|u|}{2|w|} \DENSL(w,x).
		$$
    \end{lem}

    \begin{proof}
	  Let $\SSS_N = \SSS_{L,N}(w,x)$, $\SSS'_N = \SSS_{L,N}(u,x)$, $m = |u|$ and $n = |w|$.
	  It follows from Remark \ref{RmkDensSubseq} that
	    $$
		  \DENSL(w,x) = \limsup_{N\to\infty} \frac{1}{mN}\SSS_{mN}
		      \mbox{ and }
		  \DENSL(u,x) = \limsup_{N\to\infty} \frac{1}{nN}\SSS'_{nN}.
	    $$
	It then suffices to show that for any $N$,
	    $$
		  \SSS'_{nN} \geq \frac{1}{2}(\SSS_{mN} - 1).
	    $$

	For any $N$, fix the prefix block $y = x_{[1, mnLN]}$ of length $mnLN$ in $x$.
	If $w$ begins at position $p$ in $y$, then $u$ must begin at position $p+q-1$,
	    where $u$ begins at position $q$ in $w$.
	For simplicity, only consider the first occurrence of $u$ in $w$ if necessary.
	If an occurrence of $w$ begins in the last $nL$ block,
	    it is possible that the related occurrence of $u$ does not begin in $y$.
	However, for any other occurrence of $w$, the associated occurrence of $u$
	    must begin in $y$.
	Also, it is possible for $w$ to begin in two consecutive $nL$ blocks in $y$,
	    while their corresponding beginnings of $u$ occur in the same $mL$ block.
	Therefore, it is possible to have at most two occurrences of $w$ in $nL$ blocks
	    produce at least one occurrence of $u$ in a $mL$ block.

	In the prefix block $y$,
	    we are considering $mN$ blocks of size $nL$ and $nN$ blocks of size $mL$.
	We conclude the claim and therefore the proof,
		as $\SSS_{mN}$ counts the $nL$ blocks in which $w$ begins and
		$\SSS'_{nN}$ counts the $mL$ blocks in which $u$ begins.
    \end{proof}

    \begin{lem}\label{LemNbringDens}
	  Let $w,u\in \ALPHABET^*$, $x\in \ALPHABET^\NN$ and $p_0,c\in \NN$.
	  If for every occurrence of $w$  beginning at position $p>p_0$ in $x$ there exists an occurrence of $u$ beginning at position
	      $p'$ in $x$ where $|p-p'|\leq c$, then for any $L\in \NN$
	      $$
		    \DENSL(u,x) \geq C(u,w,c,L) \DENSL(w,x),
	      $$
	  where 
		$$
		    C(u,w,c,L) = \RHScase{\frac{|u|}{3|w| + 3c/L}, & \mbox{if } |u|\leq |w|,\vspace{.05in}\\
				    \frac{L}{3}\frac{1}{3L + 6\frac{c}{|u|}},& \mbox{if } |u|>|w|}.
		$$
    \end{lem}

    \begin{rmk}
	In the case that $c = \alpha |w|$ for a real constant $\alpha$,
	  then when $|u| \leq |w|$,
	    $\DENSL(u,x) \geq C'\frac{|u|}{|w|} \DENSL(w,x)$ for some constant $C' = C'(\alpha, L)$.
	What is more interesting is that $\DENSL(u,x) \geq C'' \DENSL(w,x)$ for $C'' = C''(\alpha,L)$ when
	      $|u|>|w|$, even when $u$ is significantly longer than $w$.
    \end{rmk}

    \begin{proof}[Proof of Lemma \ref{LemNbringDens}]
	Fix $L$, and
	let $\SSS_N = \SSS_{L,N}(w,x)$, $\SSS'_N = \SSS'_{L,N}(u,x)$, $m=|u|$ and $n=|w|$
	 as in the last proof.
	 Let $\hat{c} = 3 \lceil \frac{c}{Ln}\rceil$ and $\hat{L} = \hat{c} mnL$, noting that these are
		  each positive as $c,n,m,L\geq 1$.
	We consider two cases: $m\leq n$ and $m>n$.

	If $m\leq n$, 
	consider the prefix block $y = x_{[1,N\hat{L}]}$.
	For $1\leq d \leq \hat{c}$, let
	      $$
		  \SSS_{\hat{c}mN}(d) = \sum_{\underset{j\equiv d \mathrm{mod} \hat{c}}{j=1}}^{\hat{c}mN} r_{L}(w,x,j),
	      $$
	noting that $\SSS_{\hat{c}mN} = \sum_{d=0}^{\hat{c}-1} \SSS_{\hat{c}mN}(d)$.
	Pick a $d$ such that
	      $$
		    \SSS_{\hat{c}mN}(d) \geq \frac{1}{\hat{c}}\SSS_{\hat{c}mN}.
	      $$
	For the occurrences of $w$ that contribute to $\SSS_{\hat{c}mN}(d)$,
	      at most one may fail to contribute an occurrence of $u$ in an $mL$ block due to
	      truncation\footnote{If $d=\hat{c}$, an occurrence of $w$ in the last $nL$-block may fail to produce an occurrence of $u$ in $y$.}
	      and $\lceil \frac{p_0}{nL\hat{c}}\rceil$ initial occurrences may not have an associated occurrence of $u$.
	Note that this quantity is at least one as $p_0\geq 1$.
	However, by our choices, all other occurrences must uniquely associate to an occurrence of $u$ beginning in a
	      block of length $mL$ in $y$.
	Therefore
	      $$
		   \SSS'_{\hat{c}nN} \geq \SSS_{\hat{c}mN}(d) - \left\lceil \frac{p_0}{\hat{c}nL}\right\rceil- 1
		  		\geq \frac{1}{\hat{c}} \SSS_{\hat{c}mN} - 2 \left\lceil \frac{p_0}{\hat{c}nL}\right\rceil.
	      $$
	In this case, $\DENSL(u,x) \geq \frac{m}{\hat{c}n}\DENSL(w,x)$ by Remark \ref{RmkDensSubseq},
	      as the subtracted term above is constant with respect to $N\to \infty$.
	Furthermore, $\frac{m}{\hat{c}n} > \frac{m}{3n + 3c/L}$.

	If instead $m>n$, let $\hat{b} = \lceil \frac{2Lm}{Ln(1+2\hat{c})}\rceil$,
	  where we leave $L$ in the expression for future calculations.
	We consider prefix word $y$ in $x$ of length $\hat{b}\hat{c}nmLN$ for $N\in \NN$.
	By dividing $\SSS_{\hat{b}\hat{c}mN}$ into $\hat{b}\hat{c}$ sums, we arrive at
	      $$
		      \SSS'_{\hat{b}\hat{c}nN} \geq \frac{1}{\hat{b}\hat{c}} \SSS_{\hat{b}\hat{c}mN}
		      	- 2\left\lceil\frac{p_0}{\hat{b}\hat{c}nL} \right\rceil
	      $$
	through a similar argument to the $m\leq n$ case.
	So $\DENSL(u,x) \geq \frac{m}{n\hat{b}\hat{c}}\DENSL(w,x)$.
	Note
	    $$
		  \frac{3c}{Ln}\leq \hat{c} \Rightarrow \hat{b}
		  	\leq \frac{2Lm}{Ln(1+2\hat{c})} + 1 \leq \frac{2Lm}{Ln + 6 c} + 1 = \frac{2Lm + Ln + 6c}{Ln + 6c}.
	    $$
	Combining this with the bound $\hat{c} \leq \frac{3c}{Ln} + 3 = \frac{3c + 3Ln}{Ln}$, we see that
	    $$
		  \begin{array}{rcl}
			  \frac{m}{n\hat{b}\hat{c}} & \geq & \frac{m}{n}\cdot \frac{Ln + 6c}{2Lm + Ln + 6c}
			  		\cdot \frac{Ln}{3c + 3Ln}\vspace{.05in}\\
				& = & \frac{Ln}{n}\cdot \frac{Ln + 6c}{3Ln + 3c} \cdot \frac{m}{2Lm + Ln + 6c}\vspace{.05in}\\
				& \geq & \frac{L}{1} \cdot \frac{1}{3}
					\cdot \frac{1}{2L + L \frac{n}{m} + 6\frac{c}{m}}\vspace{.05in}\\
				& \geq & \frac{L}{3}\frac{1}{3L + 6\frac{c}{m}}.
		  \end{array}
	    $$
	We have proven the result in both cases.
    \end{proof}

    \begin{lem}\label{LemSubwdLoopDens}
	Suppose for $v,\tilde{v}\in \ALPHABET^*$ the following relationships hold for $a,b\in \NN$:
		  \begin{enumerate}
			\item $|\tilde{v}| = |v| + ab$,
			\item for each $j=0,\dots, b$, $v$ begins at position $1 + ja$ in $\tilde{v}$, and
			\item each occurrence of $v$ in $x$ is contained in an occurrence of $\tilde{v}$.
		  \end{enumerate}
      Then
	    $$
		  \DENSL(v,x) \geq \frac{1}{4L + 8 \frac{a}{|v|}}\DENSL(\tilde{v},x)
		      \mbox{ and }
		  \DENSL(\tilde{v},x) \geq \frac{1}{27} \DENSL(v,x).
	    $$
    \end{lem}

    \begin{proof}
	  If we let $w = v$ and $u = \tilde{v}$, then
		      $$
			    \DENSL(\tilde{v},x) \geq \frac{1}{27}\DENSL(v,x)
		      $$
	  with $c = L|\tilde{v}|$, from Lemma \ref{LemNbringDens}.

	If $|\tilde{v}| \leq 2 |v|$, then
		      $$
			    \DENSL(v,x) \geq \frac{1}{4}\DENSL(\tilde{v},x)
		      $$
	from Lemma \ref{LemSubwdDens} with $u = v$ and $ w= \tilde{v}$.

	Now suppose $|\tilde{v}|>2|v|$.
	For fixed $N$, consider the first $NLmn$ block of $x$, $y = x_{[1,NLnm]}$ where $m = |v|$ and $n = |\tilde{v}|$.
	If $\tilde{v}$ occurs in any $Ln$ block but the last, then at least $\lfloor b/\hat{a} \rfloor$ occurrences of $v$ begin in disjoint
	      $Lm$ blocks in $y$, where
		$$
		      \hat{a} = \left\lceil \frac{Lm}{a}\right\rceil+1.
		$$
	As in the last two lemmas, occurrences in $Ln$ blocks of $\tilde{v}$ can overlap at most in pairs.
	Therefore
		$$
		      \SSS_{Nn} \geq \frac{b}{2\hat{a}} (\tilde{\SSS}_{Nm} - 1)
		$$
	where $\SSS_{Nn} = \SSS_{L,Nn}(v,x)$ and $\tilde{\SSS}_{Nm} = \SSS_{L,Nm}(\tilde{v},x)$.
	We see that $\DENSL(v,x) \geq \frac{bm}{2\hat{a}n}\DENSL(\tilde{v},x)$.
	Noting that
		$$
		 \hat{a} \leq \frac{Lm + 2a}{a},~ \frac{m}{n} < \frac{1}{2}~,\mbox{ and } ab = n-m,
		$$
	we conclude that
		$$
		  \begin{array}{rcl}
			  \frac{bm}{2\hat{a}n} & \geq & \frac{abm}{2Lmn + 4an}\vspace{.05in}\\
				    & = & \frac{nm - m^2}{2Lmn + 4an}\vspace{.05in}\\
				    & = & \frac{1 - \frac{m}{n}}{2L + 4\frac{a}{m}}\vspace{.05in}\\
				    & > & \frac{1}{4L + 8 \frac{a}{m}}
		  \end{array}	
		$$
	leaving the proof to end in a similar fashion to those in this section.
    \end{proof}

    \begin{cor}\label{CorSbwdLoopDensSmall}
	Under the conditions of Lemma \ref{LemSubwdLoopDens} with $a < \alpha |v|$ for
		   $\alpha>0$,
	    $$
		  \DENSL(v,x) \geq \frac{1}{4L + 8 \alpha}\DENSL(\tilde{v},x)
		      \mbox{ and }
		  \DENSL(\tilde{v},x) \geq \frac{1}{27} \DENSL(v,x).
	    $$
    \end{cor}

    \section{Coloring Special Rauzy Graphs}\label{SecColor}

    \subsection{Choosing $\JJJ \subseteq \NN$}

	For minimal aperiodic $\SPACE$ on finite $\ALPHABET$,
	      consider the special Rauzy graphs $\SPRGRAPH(n)$ for
	      all $n\in \NN$ with $K_\mathfrak{s}(n)$ the number of $\mathfrak{s}$-special vertices in
		each $\SPRGRAPH(n)$.	
    For $\SPACE$ satisfying \eqref{EqFixedKComp},
	we choose an infinite subset $\JJJ_0\subseteq \NN$ so that
	for $\mathfrak{s}\in \{\ell,r\}$,
	$K_\mathfrak{s} \equiv K_\mathfrak{s}(n)$ is constant for all $n\in \JJJ_0$.
      As there is a finite number of special Rauzy graphs for a given $(K_\ell,K_r)$,
	we choose infinite $\JJJ_0' \subset \JJJ_0$ so that each (unweighted)
	$\SPRGRAPH(n)$ is equivalent for all $n\in \JJJ_0'$.
    Call this common graph $\BASEGRAPH$.
    Fix a naming of the vertices in $\BASEGRAPH$
        and let $w^{(n)}$ denote the vertex in $\SPRGRAPH(n)$
	    associated to $w$ in $\Lambda$
	for all $n\in \JJJ_0'$.
    We then arrive at infinite $\JJJ \subset \JJJ'_0$ with $w^{(n)} \to \mu_{w}$ as $\JJJ\ni n\to \infty$,
	as described in Section \ref{SsecDensMeas}, for each $w\in \Lambda$.
    As such a $\JJJ$ may always be realized, we will state the desired properties as a standing assumption.
    \begin{assume}\label{ASSUMPTION}
	Consider aperiodic minimal $\SPACE$ with constant complexity growth $K$ as in \eqref{EqFixedKComp} for all $n\geq N_0$.
	We fix an infinite $\JJJ\subseteq \NN$, integers $K_\ell,K_r\leq K$,
	    an unweighted special Rauzy graph $\BASEGRAPH$,
	    and measures $\mu_w$, $w\in \Lambda$, so that
	  \begin{enumerate}\renewcommand{\theenumi}{\alph{enumi}}
	  \item $K_\mathfrak{s}(n) = K_\mathfrak{s}$ for all $n\in \JJJ$, $\mathfrak{s}\in\{\ell,r\}$,
	   \item $\SPRGRAPH(n) \equiv \BASEGRAPH$ for all $n\in \JJJ$, and
	   \item $w^{(n)} \to \mu_{w}$ as $\JJJ\ni n \to \infty$, for each $w\in \Lambda$.
	  \end{enumerate}
    \end{assume}

    \subsection{Marking $\BASEGRAPH$ with $\ERGODICsp$}
     We use the following result from \cite{cBosh1985}, which we apply to our current work.
      \begin{lem}\label{LemErgHasDens}
	  Assume \ref{ASSUMPTION} with corresponding notation.
	  Let $\nu \in \ERGODICsp$ and set $L = K+1$.
	  For each $\mathfrak{s}\in \{\ell,r\}$ there exists an $\mathfrak{s}$-special vertex $w\in \Lambda$
		    such that
		$$
		    \DENSL(\mu_{w},\nu)\geq  \frac{1}{K_\mathfrak{s}}.
		$$
      \end{lem}

      \begin{cor}\label{CorKnBd}
	  If $\SPACE$ satisfies Assumption \ref{ASSUMPTION}, then $\#\ERGODICsp\leq \min\{K_\ell,K_r\}$.
      \end{cor}

	\begin{proof}
	    By Corollary \ref{CorDensMeas}, for each $\nu\in \ERGODICsp$ there are
		  left special $u\in \Lambda$ and right special $v\in \Lambda$
		    with $\nu = \mu_u = \mu_v$.
	    Thus $\#\ERGODICsp$ is bounded above by the
		number of left special vertices and by the number of right special vertices.
	\end{proof}

      \begin{defn}\label{DefColor}
	  Under Assumption \ref{ASSUMPTION}, we \emph{mark} (or ``\emph{color}'') a vertex $w\in\BASEGRAPH$
		with $\nu \in \ERGODICsp$
	      if and only if
		    $$
			  \DENSL(\mu_{w},\nu) >0.
		    $$
	  The notation $\COLOR(w) = \nu$ means ``$w$ in $\BASEGRAPH$ is marked by $\nu$'' and
		  $\COLOR(w) = \NIL$ if we do not mark $w$.
      \end{defn}

      By Corollary \ref{CorDensMeas}, $\COLOR(w) = \nu$ implies $\nu = \mu_w$.
      So for each $w\in \Lambda$ there may be at most one $\nu \in \ERGODICsp$
		  that colors it and the above function is well-defined.

      \begin{rmk}
	    For the remainder of the paper, whenever Assumption \ref{ASSUMPTION} holds we will always set $L=K+1$ and
	      therefore will suppress it in notation for $\DENS$.
      \end{rmk}

      \begin{prop}\label{PropColorCycles}
	    Let $\JJJ$ satisfy Assumption \ref{ASSUMPTION} with graph $\BASEGRAPH$ for $\SPACE$.
	    Our coloring relation $\COLOR$ must satisfy the following rules:
		\begin{enumerate}\renewcommand{\theenumi}{\roman{enumi}}
		 \item For each $\nu\in\ERGODIC(\SPACE)$, there must exist a left special vertex $u$ and right special vertex $v$ of $\BASEGRAPH$ so that
			    $\COLOR(u) = \COLOR(v) = \nu$.
		  \item If $v$ is a right special vertex in $\BASEGRAPH$ and $\COLOR(v) = \nu$ then
			    $\COLOR(w) = \COLOR(v)$ for the unique $w$ with $w\to v$ in $\BASEGRAPH$.
			There is a vertex $w'$ with $v\to w'$ in $\BASEGRAPH$ and $\COLOR(w') = \COLOR(v)$.
		  \item If $u$ is a left special vertex in $\BASEGRAPH$ and $\COLOR(u) = \nu$ then
			    $\COLOR(w) = \COLOR(u)$ for the unique $w$ with $u\to w$ in $\BASEGRAPH$.
			There is a vertex $w'$ with $w'\to u$ in $\BASEGRAPH$ and $\COLOR(w') = \COLOR(u)$.
		\end{enumerate}
      \end{prop}

      \begin{proof}
	  (i) is simply a restatement of Lemma \ref{LemErgHasDens} using the notation here.
	  We will prove (ii) as (iii) has a similar proof.

	  Recall the fixed $x=x^{(\nu)}\in \SPACE$ that is generic for $\nu$ from Definition \ref{DefFixedGeneric}.
	  For each $n\in \JJJ$, $w^{(n)}$ occurs at most distance $Ln$ to the left of $v^{(n)}$.
	  By Lemma \ref{LemNbringDens}, $\DENS(w^{(n)},x) \geq \frac{1}{6} \DENS(v^{(n)},x)$.
	  As we assume that $\displaystyle \limsup_{\JJJ\ni n \to \infty}\DENS(v^{(n)},x)>0$, it follows that
		  $$
			\DENS(\mu_{w}, \nu) \geq \frac{1}{6}\DENS(\mu_{v},\nu) >0,
		  $$
	  giving $\COLOR(v)  = \nu$.

	  Let $y^{(1)},y^{(2)},\dots y^{(m)}$ be the vertices such that $v\to y^{(j)}$ in $\BASEGRAPH$
		for each $1\leq j \leq m$, where $m$ is the total number of edges emanating from $v$ in $\BASEGRAPH$.
	  By equation \eqref{EqPsiSpecBd}, $m\leq K+1$.
	  Choose an infinite subset $\JJJ'\subseteq \JJJ$ so that
		    $$
			  \lim_{\JJJ'\ni n \to \infty} \DENS(v^{(n)},x) = \DENS(\mu_{v},\nu).
		    $$
	  Fix $n\in \JJJ'$.
	  For each occurrence of $v^{(n)}$ in $x$, a word $y^{(j,n)}$ for some $j\in \{1,\dots,m\}$
		must occur at most $Ln$ distance to the right of $v^{(n)}$.
	  By similar reasoning to the proof of Lemma \ref{LemNbringDens},
		  $$
			\sum_{j=1}^m \DENS(y^{(j,n)},x) \geq \frac{1}{2}\DENS(v^{(n)},x).
		  $$
	  Therefore, there exists $j_n\in \{1,\dots,m\}$ with
		    $$
		      \DENS(y^{(j_n,n)},x) \geq \frac{1}{2m}\DENS(v^{(n)},x).
		    $$
	  Choose an infinite $\JJJ''\subseteq \JJJ'$ so that for some $j\in \{1,\dots,m\}$, $j_n = j$ for all $n\in \JJJ''$.
	  If $w' = y^{(j)}$, then
		      $$
			    \begin{array}{rcl}
				\DENS(\mu_{w'},\nu) &\geq& \limsup_{\JJJ''\ni n\to \infty} \DENS(y^{(j,n)},x)\vspace{.05in}\\
				      &\geq & \frac{1}{2m} \lim_{\JJJ''\ni n \to \infty} \DENS(v^{(n)},x)\vspace{.05in}\\
				    & \geq & \frac{1}{2(K+1)}\DENS(\mu_{v},\nu)\vspace{.05in}\\
				    & > & 0,
			    \end{array}
		      $$
	  or $\COLOR(w') = \nu$.
      \end{proof}
      
      \begin{cor}\label{CorQBd}
      		For each $\nu\in \ERGODICsp$ and $\BASEGRAPH$ from Assumption \ref{ASSUMPTION},
      		the set $\PRECOLOR(\nu)$ must contain a bispecial edge $u\to v$, meaning $u$ is left
		  special and $v$ is right special.
      		In particular, $\#\ERGODICsp$ is bounded by the number of bispecial edges.
      \end{cor}

      \begin{defn}
	    Under Assumption \ref{ASSUMPTION} with coloring function $\COLOR$ on $\BASEGRAPH$, let
		  $$
			\PRECOLOR(\nu) = \COLOR^{-1}(\nu)
		  $$
	    be the \emph{preimage} of $\COLOR$ for $\nu\in \ERGODICsp$; that is, 
	      the set of vertices $w\in \BASEGRAPH$ such that $\COLOR(w) = \nu$.
	    Likewise, let $\PRECOLOR(\NIL)$ denote all vertices in $\BASEGRAPH$ that
		  are not colored by $\COLOR$.
      \end{defn}

      \begin{cor}\label{CorDamronsStar}
	      If $\SPACE$ satisfies Assumption \ref{ASSUMPTION}, has constant complexity growth $K$ as in
		  equation \eqref{EqFixedKComp}
		  and either of the following hold:
		      \begin{enumerate}\renewcommand{\theenumi}{\roman{enumi}}
		       \item there exists $\nu \in \ERGODICsp$ so that $\#\PRECOLOR(\nu)>4$, or
			\item $\#\PRECOLOR(\NIL) \geq 3$,
		      \end{enumerate}
	      then $\#\ERGODICsp \leq K-2$.
      \end{cor}

      \begin{proof}
	    By definition, the sets $\PRECOLOR(\NIL)$ and $\PRECOLOR(\nu)$ for $\nu\in \ERGODICsp$ partition
		vertices of $\BASEGRAPH$, and the number of vertices is bounded by $2K$.
	    Let $E = \#\ERGODICsp$.

	    We first assume (i).
	    By Proposition \ref{PropColorCycles}, $\#\PRECOLOR(\nu) \geq 2$ for each $\nu\in \ERGODICsp$.
	    If $\nu_0\in \ERGODICsp$ has $\#\PRECOLOR(\nu_0) \geq 5$,
	   then
		      $$ 2(E-1) \leq \sum_{\ERGODICsp\ni \nu\neq \nu_0} \#\PRECOLOR(\nu) \leq 2K - \#\PRECOLOR(\nu_0) \leq 2K - 5.$$
	    Thus $E \leq K-2$ since $E$ is an integer.

	    We now assume (ii).
	    Recall that $K_\mathfrak{s}$, the number of $\mathfrak{s}$-special vertices in $\BASEGRAPH$,
		is bounded by $K$.
	    Also, if $\PRECOLOR_\mathfrak{s}$ represents $\PRECOLOR$ restricted to only
		$\mathfrak{s}$-special vertices,
		then by Proposition \ref{PropColorCycles},
		$\#\PRECOLOR_\mathfrak{s}(\nu) \geq 1$
		for all $\nu\in \ERGODICsp$.
	    Because $\#\PRECOLOR(\NIL) \geq 3$ by assumption,
		  there exists $\mathfrak{s}\in \{\ell,r\}$ so that $\#\PRECOLOR_\mathfrak{s}(\NIL) \geq 2$.
	     For this $\mathfrak{s}$,
		  $$ E \leq \sum_{\nu\in \ERGODICsp} \#\PRECOLOR_\mathfrak{s}(\nu) \leq K - \#\PRECOLOR_\mathfrak{s}(\NIL) \leq K-2,$$
	      as desired.
	    \end{proof}

      \section{Bispecial Moves}\label{SecBiSpMoves}

      \subsection{Bispecial Words from $\RGRAPH(n)$ to $\RGRAPH(n+1)$}

      For a language $\LANGUAGE$ satisfying the binary extension condition,
	          we now consider the types of bispecial words described in Section \ref{Ssec2EC} and
	    explore the appropriate transition from Rauzy graph $\RGRAPH(n)$ to
	    Rauzy graph $\RGRAPH(n+1)$.
      For a bispecial $w \in \LANGUAGE(n)$, let $a,a',b,b'\in \ALPHABET$, $a\neq a'$ and $b\neq b'$, be the letters such that
	    $aw,a'w,wb,wb'\in \LANGUAGE(n+1)$.
	\begin{itemize}
	  \item If $w$ is \emph{weak bispecial},
	      the set of two-way extensions $\EXTENDbs(w)$ consists of exactly two words,
		$awb$ and $a'wb'$, up to appropriate naming of $a,a',b,b'$.
	      The transition about $w$ from $\RGRAPH(n)$ to $\RGRAPH(n+1)$ is given in Figure \ref{FigBSMa}.
	  \item If $w$ is \emph{strong bispecial}, then $awb,awb',a'wb,a'wb'\in \LANGUAGE$.
		In particular, both $aw$ and $a'w$ are right special while
		    both $wb$ and $wb'$ are left special.
		The transition from $n$ to $n+1$ is given in Figure \ref{FigBSMc}.
	  \item If $w$ is \emph{regular bispecial}, then $awb,awb',a'wb\in \LANGUAGE$ and
		    $a'wb'\notin\LANGUAGE$, up to renaming the letters.
		See Figure \ref{FigBSMb} for the transition.
	  \end{itemize}
	
      \begin{figure}[t]
	    \begin{center}
		\subfigure[Weak bispecial]{\label{FigBSMa}
		      \includegraphics[width=.6\textwidth]{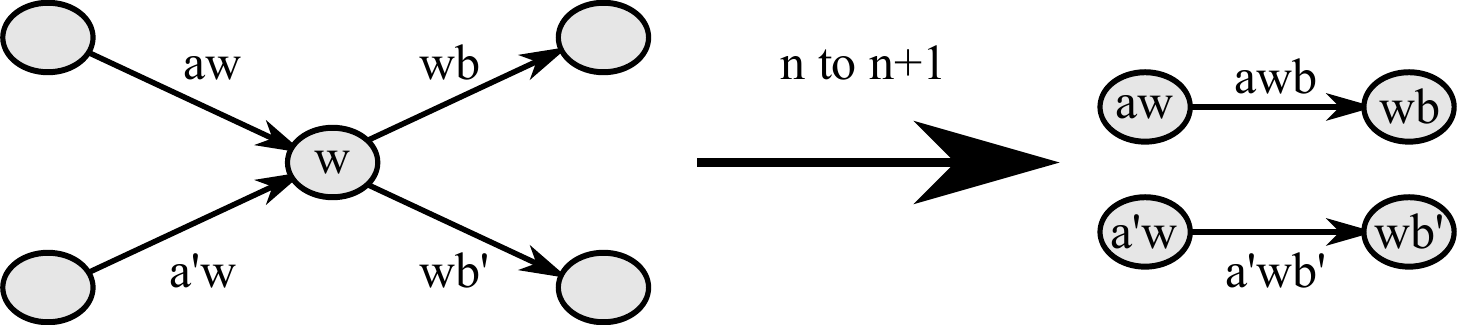}}
	    \end{center}
	    \begin{center}
		\subfigure[Regular bispecial]{\label{FigBSMb}
		      \includegraphics[width=.6\textwidth]{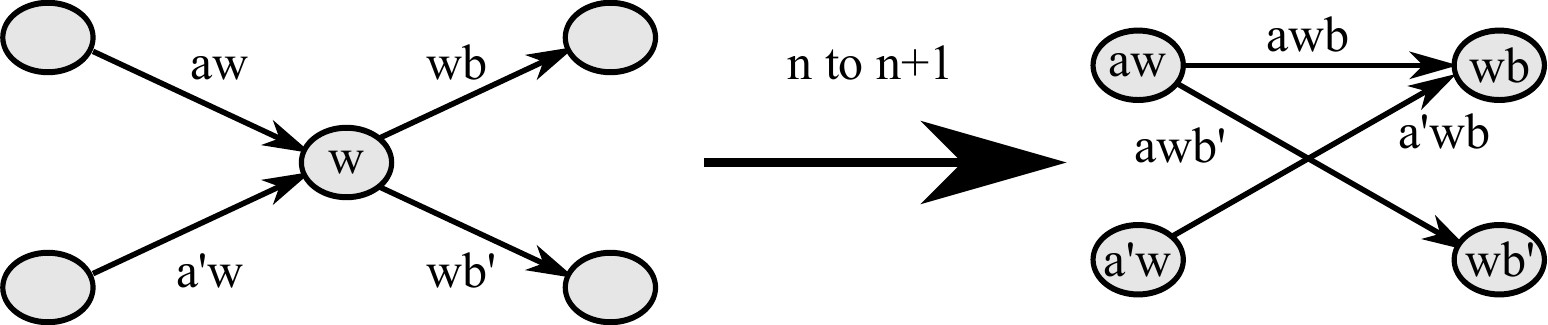}}
	    \end{center}
	    \begin{center}
		\subfigure[Strong bispecial]{\label{FigBSMc}
		      \includegraphics[width=.6\textwidth]{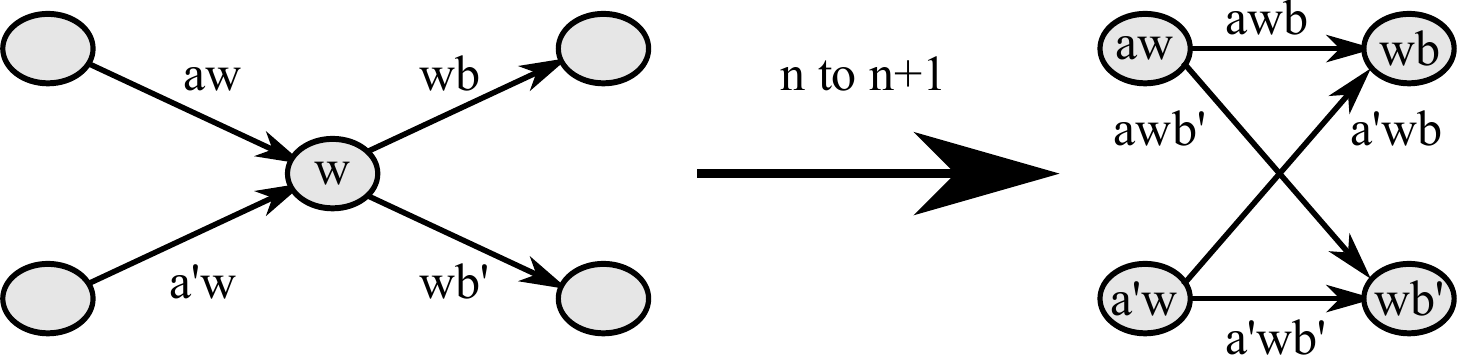}}
	    \end{center}
	    \caption{The three types of bispecial moves on $w$ from Rauzy graph $\RGRAPH(n)$ to Rauzy graph $\RGRAPH(n+1)$.}\label{FigBSM}
      \end{figure}

	Naturally, we would like to describe the transition for the other words in $\LANGUAGE(n)$.
	If $v$ is right (and not left) special, then its unique left extension $av$ is also right special.
	Likewise if $u$ is left (and not right) special, then its unique right extension $ub$ is also left special.
	If $w$ is not special, then its left extension $aw$ is not right special and its right extension $wb$ is
		not left special.
	We conclude that a unispecial word in $\LANGUAGE(n)$ associates uniquely to a unispecial word in $\LANGUAGE(n+1)$.
	However, a bispecial word in $\LANGUAGE(n)$ may associate to zero, one or two special words of each type in
		$\LANGUAGE(n+1)$, depending on the nature of the bispecial word.
	Furthermore, all special words in $\LANGUAGE(n+1)$ must be associated to special words in $\LANGUAGE(n)$ as described
		  here.

	\begin{rmk}
	      If $\LANGUAGE$ does not satisfy the binary extension condition,
		then most of the observations in this section still hold.
	      In particular, there remains a well-defined association between special words
		      in $\LANGUAGE(n+1)$ and those in $\LANGUAGE(n)$.
	      The behavior of a bispecial word $w$ will vary depending on the
		      nature of $\EXTENDbs(w)$.
	      However, there are many possible outcomes.
	      For example, if $w\in \LANGUAGE$ is bispecial such that
		      $aw,wb\in \LANGUAGE$ for all $a,b\in \ALPHABET$, where $\#\ALPHABET> 2$, then
		  $$ \#\ALPHABET \leq \#\EXTENDbs(w) \leq \big(\#\ALPHABET\big)^2.$$
	      To further complicate matters, the local transition from $\RGRAPH(n)$ to $\RGRAPH(n+1)$ is no longer
		    uniquely determined by the value $\#\EXTENDbs(w)$.
	      This is why we typically consider $\LANGUAGE$ with the binary extension condition
		    in detail for the rest of the paper and end with discussions
		    for more general languages.
	\end{rmk}

      \subsection{From $\SPRGRAPH(n)$ to $\SPRGRAPH(n+1)$}
	Now consider the transition from special Rauzy graph $\SPRGRAPH(n)$
		  to special Rauzy graph $\SPRGRAPH(n+1)$.
	If $w$ is a unispecial vertex in $\SPRGRAPH(n)$,
	    then we name $w$ in $\SPRGRAPH(n+1)$ its unique special extension.
	We see that we only need to consider bispecial words of length $n$
	    in order to determine the structure of $\SPRGRAPH(n+1)$ given $\SPRGRAPH(n)$.
	Before we do so, we will briefly note the relationship between $\rho_n(w\to w')$,
	  the weight of edge $w\to w'$ in $\SPRGRAPH(n)$, and $\rho_{n+1}(w\to w')$
	  when both $w$ and $w'$ are unispecial.
	  \begin{itemize}
		\item If $w$ and $w'$ are either both left special or both right special,
			then $$\rho_{n+1}(w\to w') = \rho_n(w\to w').$$
		\item If $w$ is right special and $w'$ is left special, then
			      $$ \rho_{n+1}(w\to w') = \rho_n(w\to w') +1.$$
		\item If $w$ is left special and $w'$ is right special; that is,
			$w\to w'$ is a bispecial edge, then
			      $$ \rho_{n+1}(w\to w') = \rho_n(w\to w') -1.$$
	  \end{itemize}
	It follows that if $\SPRGRAPH(n)$ has no edges of weight $0$ (or equivalently,
	    there are no bispecial words in $\LANGUAGE(n)$), then
	    $\SPRGRAPH(n+1) \equiv \SPRGRAPH(n)$, and only the bispecial edges
	    decrease in weight.
	In fact, the special graphs will remain equivalent until a bispecial edge
	      decreases to weight $0$ and is associated to a bispecial word in
	      $\LANGUAGE$.

	We begin by assuming that $\LANGUAGE$ has the binary extension condition
	      for $N_0$ and $n\geq N_0$.
	For now, assume that $w$ is the only bispecial word of length $n$.
	Recall that $w$ is actually represented by two vertices $u$ and $v$ in $\SPRGRAPH(n)$,
	      where $u$ is left special while $v$ is right special.

	If $w$ is strong bispecial, then there are four vertices in $\SPRGRAPH(n+1)$ associated to $w$.
	    We denote the left extensions (which are \emph{right special}) as $v$ and $v'$,
	      where the choice between the two will be made when needed.
	Likewise, we name the right extensions by $u$ and $u'$,
	    as they are the resulting left special vertices from the transition.
	We call this change a \emph{strong bispecial (SBS) move} on edge $u\to v$.

	If instead $w$ is regular bispecial, then we name the unique right extension that is left special
		$u$ in $\SPRGRAPH(n+1)$ and we name the unique left extension that is right special $v$.
	Note in this case that the other extensions are not vertices in $\SPRGRAPH(n+1)$.
	We call this change a \emph{regular bispecial (RBS) move} on edge $u \to v$.

	If $w$ is weak bispecial, then no extensions will be vertices in $\SPRGRAPH(n+1)$.
	  In this case, the surrounding associated special words will be connected by edges directly.
	We call this change a \emph{weak bispecial (WBS) move} on edge $u \to v$.
	\begin{figure}[t]
	      \begin{center}
		  \includegraphics[width = .7\textwidth]{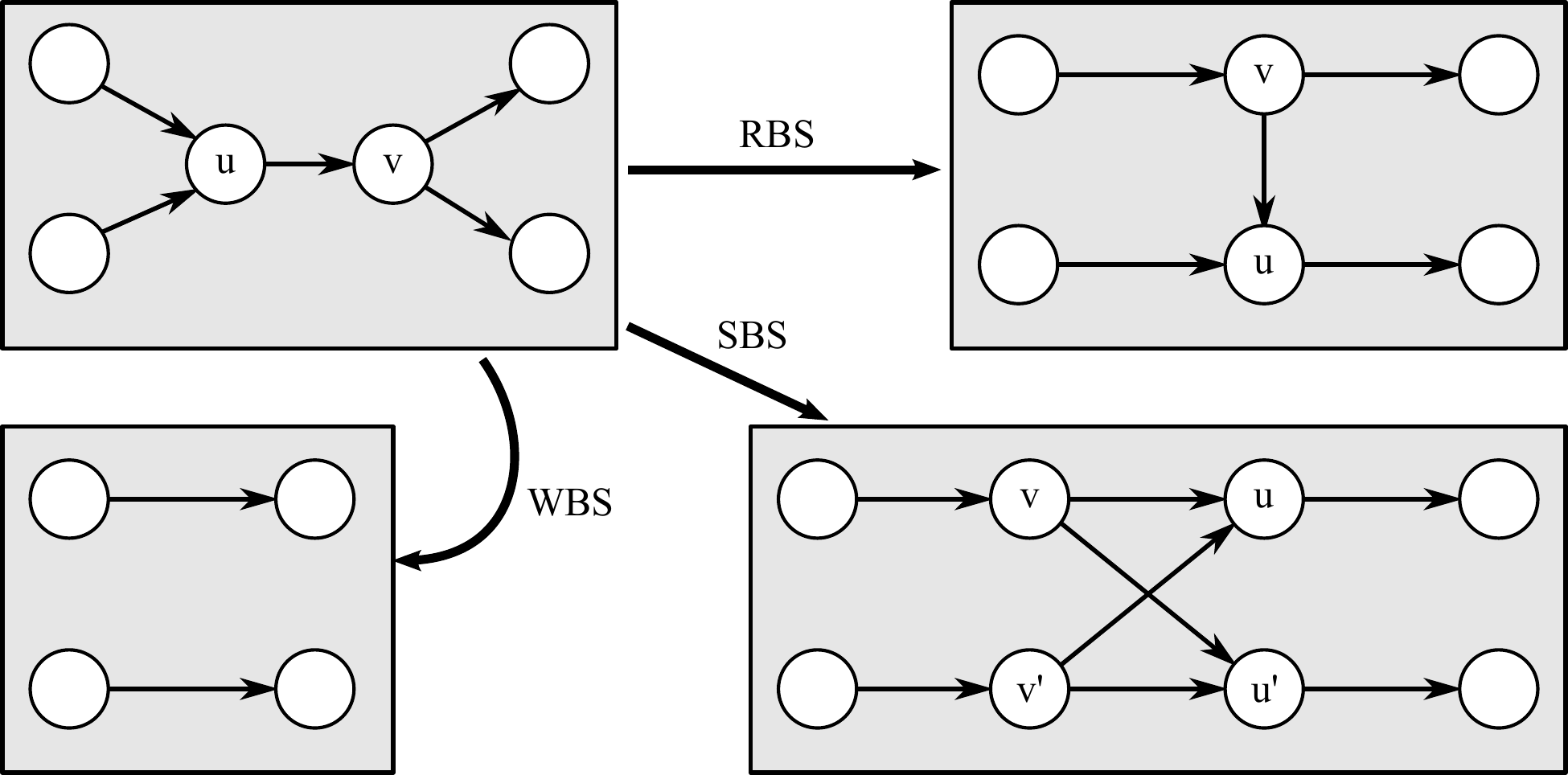}
	      \end{center}
	      \caption{The three types of bispecial moves from $\SPRGRAPH(n)$ to $\SPRGRAPH(n+1)$ about bispecial word $u=v$.
			The outer four vertices represent the
			 attached vertices in the represpective special graphs.}\label{FigAllBS}
	\end{figure}

	Each possible move is given in Figure \ref{FigAllBS}.
	These moves are all illustrated in Figures \ref{FigBSMa}-- \ref{FigBSMc}.
	Often, multiple bispecial words exist of a given length $n$.
	The following lemma tells us that we may realize the transition
	    from $\SPRGRAPH(n)$ to $\SPRGRAPH(n+1)$ (as unweighted graphs) by applying each individual
	  	bispecial move one at a time, in any order we choose.
	While the weights are not claimed to be given by this realization (see Remark \ref{RmkNasty}),
		the difference will not affect the coloring from one graph to the next as
		discussed in the next section.
	\begin{lem}
	      If $u^{(j)} \to v^{(j)}$, $1\leq j \leq m$, are the edges in $\SPRGRAPH(n)$ representing the
		bispecial words $w^{(1)},\dots,w^{(m)}$ of length $n$, then the (unweighted) graph
		$\SPRGRAPH(n+1)$ is obtained by
		applying the associated bispecial move on each edge $u^{(j)} \to v^{(j)}$ one at a time in the order
		$j=1,\dots,m$.
	\end{lem}

	\begin{proof}
		Let $\ALPHABET_1 = \LANGUAGE(n)$ and $\pi_n : \ALPHABET^* \to \ALPHABET_1^*$ be given by
		      $\pi_n(w) = \EMPTYWORD$ (the empty word) if $|w|<n$ and
			$$ \big(\pi_n(w)\big)_i = w_{[i,i+n-1]}\mbox{ for } 1\leq i \leq |w| - n +1$$
		if $|w| \geq n$.
		Then $\LANGUAGE_1 = \pi_n(\LANGUAGE)$ is well-defined and $\pi_n$ is an isomorphism
		    from $\RGRAPH(k+n-1)$ to $\RGRAPH_1(k)$ (the Rauzy graph given by $\LANGUAGE_1$)
		    for each $k\geq 1$.
		This implies that $\SPRGRAPH_1(k) \equiv \SPRGRAPH(k+n-1)$ for all $k\geq 1$,
		    with equal edge weights.

	      Let $m_\mathfrak{s} \geq m$, $\mathfrak{s}\in \{\ell,r\}$, be the number
		      of $\mathfrak{s}$-special elements in $\LANGUAGE_1(1)$.
		Order the left special words of $\LANGUAGE_1(1)$ as $U^{(1)},\dots, U^{(m)\ell)}$
		    and the right special words as $V^{(1)},\dots,V^{(m_r)}$ so that
		      $W^{(j)} := U^{(j)} = V^{(j)}$ is the image of $w^{(j)}$ by $\pi_n$, $1\leq j \leq m$.
		If
			$$ f(j) = \RHScase{2m + j, & 1\leq j \leq m, \\
				    7m, & m< j,}$$
		for $1\leq j \leq m_\ell$ and $1\leq i \leq f(j)$, let $x_j^i$
		    be distinct letters that do not belong to $\ALPHABET_1$.
	      Likewise, for $1 \leq j \leq m_r$ and $1\leq i \leq f(j)$ let $y_j^i$ be
			distinct letters that do not belong to $\ALPHABET_1$ or equal to any of the $x_j^i$'s.
		Then let
		  $$\ALPHABET_2 = \ALPHABET_1 \cup \{x_j^i: 1\leq j \leq m_\ell,~1 \leq i \leq f(j)\}
			  \cup \{y_j^i: 1\leq j \leq m_r,~1 \leq i \leq f(j)\}.$$
		Define the substitution $\sigma:\ALPHABET_1^* \to \ALPHABET_2^*$ by the rule
			$$ \sigma(W) = \RHScase{x_j^1\dots x_j^{f(j)} W y_j^1\dots y_j^{f(j)}, & W = W^{(j)},~1\leq j\leq m,\\
					  x_j^1\dots x_j^{f(j)} W, & W = U^{(j)},~m<j\leq m_\ell,\\
					  W y_j^1\dots y_j^{f(j)}, & W = V^{(j)},~m<j\leq m_r,\\
					  W, \mbox{otherwise},}$$
		  for each $W\in \LANGUAGE_1(1)$ and extend by concatenation.
		Let $\LANGUAGE_2$ be the language generated by the image of $\LANGUAGE_1$ under $\sigma$,
			meaning $W'\in \LANGUAGE_2$ iff it is a subword of $\sigma(W)$ for some $W\in \LANGUAGE_1$.
		
	      Let $\Pi:\LANGUAGE_2 \to \LANGUAGE_1$ be defined as follows: $\Pi(W) = Y$ where $Y$ is the
		      unique word of minimal length so that $W$ is a subword of $\sigma(Y)$.
	      Therefore, for any proper subword $Y'$ of $Y = \Pi(W)$,
		      $$ |\sigma(Y')|< |W| \leq |\sigma(Y)|.$$
		We make the following claims for $W$ and $Y = \Pi(W)$:
		      \begin{itemize}
		       \item $W\in \LANGUAGE_2$ is left special iff $W$ is a prefix of left special $Y$,
		       \item $W\in \LANGUAGE_2$ is right special iff $W$ is a suffix of right special $Y$,
			\item $W\in \LANGUAGE_2$ is bispecial iff $W = \sigma(Y)$ and $Y$ is bispecial,
			\item for $1\leq k \leq  6m+1$, $W\in \LANGUAGE_2(k)$ is bispecial iff
					  $Y = W^{(j)}$, $1\leq j \leq m$ and $|W| = 4m + 2j +1$.
		      \end{itemize}
		It follows that $\SPRGRAPH_2(1) \equiv \SPRGRAPH_1(1) \equiv \SPRGRAPH(n)$
		    and for each $1\leq j \leq m$, the change from $\SPRGRAPH_2(4m+2j+1)$ to
		    $\SPRGRAPH_2(4m+2j+2)$ is realized by exactly one bispecial move
		    and the move is by the to the one given by the edge $u^{(j)}\to v^{(j)}$
		    from the original special graph $\SPRGRAPH(n)$ to $\SPRGRAPH(n+1)$.

		We finish with the claim $\SPRGRAPH_2(6m+2) \equiv \SPRGRAPH_1(2)$ as unweighted graphs.
		To do so, we define bijections $\Xi_\mathfrak{s}$, $\mathfrak{s}\in \{\ell,r\}$
			that assigns to each $\mathfrak{s}$-special word in $\LANGUAGE_2(6m+2)$
			an $\mathfrak{s}$-special word in $\LANGUAGE_1(2)$.
		We fully and define $\Xi_\ell$ and prove its bijectivity, as the case for $\Xi_r$ follows
			by analogy.

		For left special $\tilde{U}\in \LANGUAGE_2(6m+2)$, let $Y = Y_1\dots Y_k = \Pi(\tilde{U})$.
		If $|Y|\geq 2$, let $\Xi_\ell(\tilde{U}) = Y_1Y_2$.
		If $|Y| = 1$, then $Y = Y_1 = U^{(j)}$, $m<j\leq m_\ell$, as $|Y_1| \geq 6m+2$.
		In particular, $Y_1$ is not right special and must admit a unique right extension $Y_1Z_1$ that
		      is also left special. In this case $\Xi_\ell(\tilde{U}) = Y_1Z_1$.

		We claim $\Xi_\ell$ is injective.
		Suppose $\Xi_\ell(\tilde{U}) = \Xi_\ell(\hat{U})$.
		If $|\Pi(\tilde{U})| \leq 2$,
			      then it follows that $\Pi(\hat{U}) = \Pi(\tilde{U})$ and therefore $\tilde{U} = \hat{U}$.
		Suppose by contradiction that $\tilde{U} \neq \hat{U}$ and $\Pi(\tilde{U}) \neq \Pi(\hat{U})$
			with $|\Pi(\tilde{U})|,|\Pi(\hat{U})|\geq 3$.
		If
			$$ \Pi(\tilde{U}) = Y_1\dots Y_k
			      \mbox{ and }
			    \Pi(\hat{U}) = Y'_1\dots Y'_{k'}$$
		then there exists $2\leq k_1 <k,k'$ so that
		      $Y_1\dots Y_{k_1} = Y'_1\dots Y_{k_1}'$.
		However this implies that $Y_1\dots Y_{k_1}$ is bispecial and in particular $Y_k$ is right special.
		In this case
			$$ |\sigma(Y_1\dots Y_{k_1})| = |\sigma(Y_1)| + |\sigma(Y_{k_1})| \geq 2[4m+3] > 6m+3,$$
		contradicting the definition of $\Pi(\tilde{U})$.
		 
		We now show that $\Xi_\ell$ is surjective.
		Let $Y\in \LANGUAGE_1(2)$ be left special and consider its maximal unique right extension
		    $YZ$.
		Because either $Y$ is itself bispecial (and so $Z = \EMPTYWORD$, the empty word)
		      or $YZ$ is bispecial, it follows that $YZ$ ends with a right special word.
		Therefore $|\sigma(YZ)| \geq 2[4m+3] >6m+3$.
		If $U = \sigma(YZ)_{[1,6m+2]}$,
		    $\Pi(U)$ is a prefix of $YZ$ and so $\Xi_\ell(U) = Y$.

		The maps $\Xi_\ell$ and $\Xi_r$ imply a bijection from edge words $W \to W'\in \SPRGRAPH_2(6m+2)$
		      to edge words $\Xi_\mathfrak{s}(W) \to \Xi_{\mathfrak{s}'}\in \SPRGRAPH_1(2)$,
		      where $W$ and $W'$ are $\mathfrak{s}$-special and $\mathfrak{s}'$-special
		      vertices respectively.
	\end{proof}

	\begin{rmk}\label{RmkNasty}
	      While separating $m$ simultaneous bispecial moves into $m$ steps does not
		    yield a graphs with equal edge weights, it may be shown that
		    $$|\rho_2(W\to W') - \rho_1\big(\Xi_\mathfrak{s}(W) - \Xi_\mathfrak{s}(W')\big)| \leq 7m^2,$$
		where these objects are defined in the previous proof.
	      Here $\rho_1$ gives the edge weights for $\SPRGRAPH_1(2)$ and $\rho_2$ gives the edge weights for $\SPRGRAPH_2(6m+2)$.
	      For example, if $W\to W'$ is a bispecial edge, then
			$\rho_1\big(\Xi_\ell(W) \to \Xi_r(W')\big)$ is equal to $\rho_2(W \to W')$ minus
		      the appearances of $x_j^i$'s and $y_j^i$'s
		      in the edge word $W \to W'$.
	      Because $K$ bounds $m$, this difference is
		    small for large $n$.
	      Therefore, we may extend results such as Lemma \ref{LemColorMoves} below,
		    which addresses one bispecial move from $\SPRGRAPH(n)$ to $\SPRGRAPH(n+1)$,
		    to the case of simultaneous bispecial moves.
	\end{rmk}

	If $\LANGUAGE$ does not satisfy the binary extension condition,
	      the principles in this section still apply.
	For example, the special graphs $\SPRGRAPH(n)$ remain the same as $n$ changes unless a bispecial edge's weight decreases to 0.

	\subsection{Finding $\BASEGRAPH'$ from $\BASEGRAPH$}
	As indicated in the previous section,
	     $$\SPRGRAPH(n) \equiv \SPRGRAPH(n+1)$$
	    if
	  no bispecial words of length $n$ exist in $\LANGUAGE(n)$.
	\begin{defn}
	    For fixed $\LANGUAGE$ and any $n\in \NN$, let the \emph{next bispecial value} for $n$ be
		  $$
			\BISPECIAL_n := \min\{n'\geq n: \LANGUAGE(n') \mbox{ contains a bispecial word}\}.
		  $$
	    We call the set of \emph{bispecial values} for $\LANGUAGE$
		  $$
			\BISPECIAL = \{\BISPECIAL_n: n\in \NN\}.
		  $$
	    The number of \emph{bispecial steps} from $m$ to $m'>m$ is given by
		  $$
			\#\left(\BISPECIAL \cap [m,m')\right).
		  $$
	\end{defn}
	
	Suppose we have $\JJJ$ with $\BASEGRAPH$ from Assumption \ref{ASSUMPTION}.
	For each $n\in \JJJ$, let $n' = \BISPECIAL_n+1$.
	Then $\SPRGRAPH(\tilde{n})\equiv \Lambda$ for $n\leq \tilde{n} < n'$
	      and $K_\ell(n'),K_r(n')\leq K$.
	We may therefore choose an infinite set $\JJJ'$ of such $n'$ values and a special graph $\BASEGRAPH'$
		  so that $\SPRGRAPH(n') \equiv \BASEGRAPH'$.
	Therefore $\JJJ'$ and $\BASEGRAPH'$ satisfy Assumption \ref{ASSUMPTION} parts (a) and (b).
	 By passing to another infinite subsequence, $\JJJ'$ will satisfy (c) as well.
	 We will define a coloring function $\COLOR'$ on $\BASEGRAPH'$.
	 Because we will want to relate $\COLOR$ on $\BASEGRAPH$ to $\COLOR'$ on $\BASEGRAPH'$,
	    we then reduce $\JJJ$ so that the map $n \mapsto \BISPECIAL_n +1$ is a
	    bijection from $\JJJ$ to $\JJJ'$.
	
		Because we are replacing $\JJJ$ with a subset, it is possible that
		  $\COLOR(w)$ will now be $\NIL$ when it was initially $\nu\in \ERGODICsp$,
		  as $\COLOR$ depends on $\JJJ$.
	      To prevent the loss of color when producing new subsequences, we amend Assumption \ref{ASSUMPTION}
		  so that $\COLOR$ will be preserved when reducing $\JJJ$.
	      \begin{assume}\label{ASSUMPTION2}
		     Consider $\SPACE$, $\JJJ$, $\BASEGRAPH$, $K_\ell$, $K_r$ that satisfy Assumption \ref{ASSUMPTION}.
		      By replacing $\JJJ$ with a subsequence, the original conditions
			(a)--(c) from Assumption \ref{ASSUMPTION} hold
		      and furthermore
			  \begin{enumerate}\renewcommand{\theenumi}{\alph{enumi}}
					      \setcounter{enumi}{3}
			   \item For each $w\in \BASEGRAPH$ with representative $w^{(n)}\in \SPRGRAPH(n)$ for all $n\in \JJJ$
				      and for each $\nu\in \ERGODICsp$,
				  $$\lim_{n\to\infty}\DENS(w^{(n)},x) = \DENS(\mu_w,\nu)$$
				  where $x$ is the generic point for $\nu$ from Definition \ref{DefFixedGeneric}.
			  \end{enumerate}
	      \end{assume}
	      In other words, because $\DENS(\mu_w,\nu):= \DENSL(\mu_w,\mu)$ is defined via a limsup,
		  it may now be realized as a limit.
	      Under this new assumption, $\COLOR$ will not change if $\JJJ$ is ever restricted
			to a subset.
	      Note that $\#\ERGODIC(\SPACE)<\infty$ under Assumption \ref{ASSUMPTION} by Corollary \ref{CorQBd}.
	      Therefore, Assumption \ref{ASSUMPTION2} may be used whenever Assumption
		    \ref{ASSUMPTION} holds.

	\begin{defn}
	    The new set $\JJJ'$ with corresponding data is the result of \emph{one bispecial step}
	      from $\JJJ$.
	    For any $M\in \NN$, we may analogously define $\JJJ^{(M)}$ that satisfies Assumption \ref{ASSUMPTION2}
		and is the result of \emph{$M$ bispecial steps} from $\JJJ$ by choosing
		  each $\JJJ^{(m+1)}$ to be one bispecial step from $\JJJ^{(m)}$ for each $m<M$.
	\end{defn}

	\begin{rmk}
	    As we shall see soon, it is possible to have $\BASEGRAPH' = \BASEGRAPH$ or even
		$\BASEGRAPH^{(M)} = \BASEGRAPH^{(M-1)} = \dots = \BASEGRAPH' = \BASEGRAPH$.
	\end{rmk}

	\begin{lem}\label{LemColorMoves}
	    Consider $\SPACE$ with language $\LANGUAGE$ satisfying the binary extension condition for $N_0$.
	    Suppose $\BASEGRAPH$ is from Assumption \ref{ASSUMPTION2} and $\BASEGRAPH'$ is the result of one bispecial step.
	    Let $\COLOR$ denote the marking function on $\BASEGRAPH$ and $\COLOR'$ denote the marking function on $\BASEGRAPH'$.
	    \begin{enumerate}\renewcommand{\theenumi}{\roman{enumi}}
	     \item If $w\in \BASEGRAPH$ is not an endpoint of an edge participating in a bispecial move, then $w \in \BASEGRAPH'$ and $\COLOR(w) = \COLOR'(w)$.
	      \item Suppose $u\to v$ in $\BASEGRAPH$ is changed by an SBS move with corresponding left special vertices $u,u'$
			  and right special vertices $v,v'$ in $\BASEGRAPH'$.
		    Then
			$$
			      \{\COLOR(u)\}\subseteq \{\COLOR'(u),\COLOR'(u')\} \subseteq \{\COLOR(u),\NIL\}
			$$
		      and
			$$
			      \{\COLOR(v)\}\subseteq \{\COLOR'(v),\COLOR'(v')\} \subseteq \{\COLOR(v),\NIL\},
			$$
		    where we recall that $\COLOR(v) = \COLOR(u)$.
	      \item Suppose $u\to v$ in $\BASEGRAPH$ is changed by an RBS move to get $u$ and $v$ in $\BASEGRAPH'$.
		    Then
			  $$
			      \{\COLOR(u)\}\subseteq \{\COLOR'(u),\COLOR'(v)\} \subseteq \{\COLOR(u),\NIL\}.
			  $$
	      \item Suppose $u\to v$ in $\BASEGRAPH$ is changed by a WBS move to $\BASEGRAPH'$.
		    Then for the four vertices $w_1,\dots,w_4$ in $\BASEGRAPH'$ that are connected by the edges made
			    by $u\to v$ from $\BASEGRAPH$,
			  $$
			      \{\COLOR(u)\}\subseteq \{\COLOR'(w_1),\COLOR'(w_2),\COLOR'(w_3),\COLOR'(w_4)\}
			      	\subseteq \{\COLOR(u),\NIL\}.
			  $$
	    \end{enumerate}
	\end{lem}

	\begin{proof}
	
	 	Note that for any $\SPACE$ satisfying equation \eqref{EqFixedKComp} for $n\geq N_0$,
		    $$
		      \BISPECIAL_n - n < (K+1)n
		    $$
		  for all large enough $n$,
		      as any bispecial edge in $\SPRGRAPH(n)$ has weight at most $\COMPLEXITY(n+1)$,
		      the number of edges in Rauzy graph $\RGRAPH(n)$.
		Consider for a vertex $w\in \BASEGRAPH$ or $w\in \BASEGRAPH'$ the associated
		    words $w^{(n)} \in \LANGUAGE(n)$ for $n\in \JJJ$ and $w^{(n')}\in \LANGUAGE(n')$
		    for $n' \in \JJJ$ as appropriate.

		We will first prove (i).
		  For large $n\in \JJJ$ with $n'= (\BISPECIAL_n+1)\in \JJJ'$ we apply Lemma \ref{LemSubwdDens} to see that because $w^{(n)}$ is a subword of $w^{(n')}$,
			  $$
				\DENS(\mu_w ,\nu) \geq \frac{1}{2(K+1)}\DENS(\mu'_{w},\nu),
			  $$
		    where $\mu'_w$ is the measure associated to $w\in \BASEGRAPH'$ from Assumption \ref{ASSUMPTION}.
		   Likewise if we apply Lemma \ref{LemNbringDens}\footnote{While the bounding constant $1/27$ seen here matches that
			    in Corollary \ref{CorSbwdLoopDensSmall},
			    that result cannot be applied as $w'$ is not of the required form.}
			noting that $|w^{(n)}| < |w^{(n')}|$ with $c = Ln$ we
			  have
			  $$
			      \DENS(\mu'_w, \nu) \geq \frac{1}{27} \DENS(\mu_w,\nu).
			  $$
		  Therefore $\DENS(\mu_w,\nu)>0$ if and only if $\DENS(\mu_w',\nu)>0$ and so $\COLOR(w) = \COLOR'(w)$.

		We show (ii) for the vertices $u$ and $u'$, as the other relationship has a very similar proof. Furthermore, the proofs of (iii) and (iv) are similar so we omit them.
		  We may again apply Lemma \ref{LemSubwdDens} to see that if $\COLOR'(u) = \nu$ then $\COLOR(u) = \nu$,
		      as $u^{(n)}$ is a subword of $u^{(n')}$.
		Likewise, if $\COLOR'(u') = \nu'$ then $\COLOR(u) = \nu'$.
		Therefore, $\COLOR'(u)$ and $\COLOR'(u')$ may only take values in the set
			$\{\COLOR(u),\NIL\}$.

		Now suppose $\COLOR(u)=\COLOR(v) = \nu$.
		For each large $n$,
		   let $w^{(1,n+1)},w^{(2,n+1)}\in \LANGUAGE(n+1)$ be the right extensions of $v^{(n)}$.
		These may be uniquely extended to the left until length $n'$, and these are precisely the right extensions
		    of bispecial word $u^{(n'-1)} = v^{(n'-1)}$; that is, the words in $\LANGUAGE(n')$ that relate to $u,u'\in \BASEGRAPH'$.
		Following the proof of Proposition \ref{PropColorCycles}, there exists $j\in \{1,2\}$ so that
			  $$ \DENS(w^{(j,n+1)}, x) \geq \frac{n}{2(n+1)} \DENS(v^{(n)},x).$$
		By Lemma \ref{LemNbringDens}, there exists $\tilde{w}\in \BASEGRAPH'$ such that
			  $$ \DENS(\mu_{\tilde{w}}', \nu) \geq \frac{1}{2(K+2)}\DENS(\mu_v,\nu),$$
		  where $\tilde{w}$ is either $u$ or $u'$ depending on which $k$ satisfies the previous inequality
		      for infinitely many $n$.
		Therefore either $\COLOR'(u) = \nu$ or $\COLOR'(u') = \nu$,
			and the remaining containment has been shown.
	\end{proof}

	\begin{rmk}
	      If the language $\LANGUAGE$ does not satisfy the binary extension
		      condition,
	      then Lemma \ref{LemColorMoves} will still follow by a similar proof.
	      However, the wording will become more complicated.
	\end{rmk}

	\begin{rmk}
	    If $\BASEGRAPH,\BASEGRAPH',\dots, \BASEGRAPH^{(M)}$ are increments of $M$ bispecial steps,
	      then coloring $\COLOR^{(M)}$ on $\BASEGRAPH^{(M)}$ 
		  will be related to the coloring $\COLOR$ on $\BASEGRAPH$
	      by iteratively applying Lemma \ref{LemColorMoves}.
	\end{rmk}

	\subsection{Minimal preimages of $\COLOR$}\label{SsecLoops}

	Consider a shift $\SPACE$ with language $\LANGUAGE$ that satisfies the binary extension condition.
	From Proposition \ref{PropColorCycles}, for each $\nu\in \ERGODICsp$,
		  the preimage set
		$
		      \PRECOLOR(\nu) = \COLOR^{-1}(\nu),
		$
	contains a right special and a left special vertex in $\BASEGRAPH$.
	Here, we consider the case $\#\PRECOLOR(\nu) = 2$.
	$\PRECOLOR(\nu)$ must equal $\{u,v\}$, where $u$ is left special and $v$
		is right special.
	It must also be that $u\to v$ and $v\to u$; otherwise, $\PRECOLOR(\nu)$ would contain more than two vertices.
	We conclude that $u$ and $v$ form a loop in $\BASEGRAPH$ as in Figure \ref{FigLoop1},
	      where $w,z\in \BASEGRAPH$ represent the adjacent vertices with $w\to u,v\to z$.
	\begin{figure}[t]
	    \begin{center}
		    \includegraphics[width=.5\textwidth]{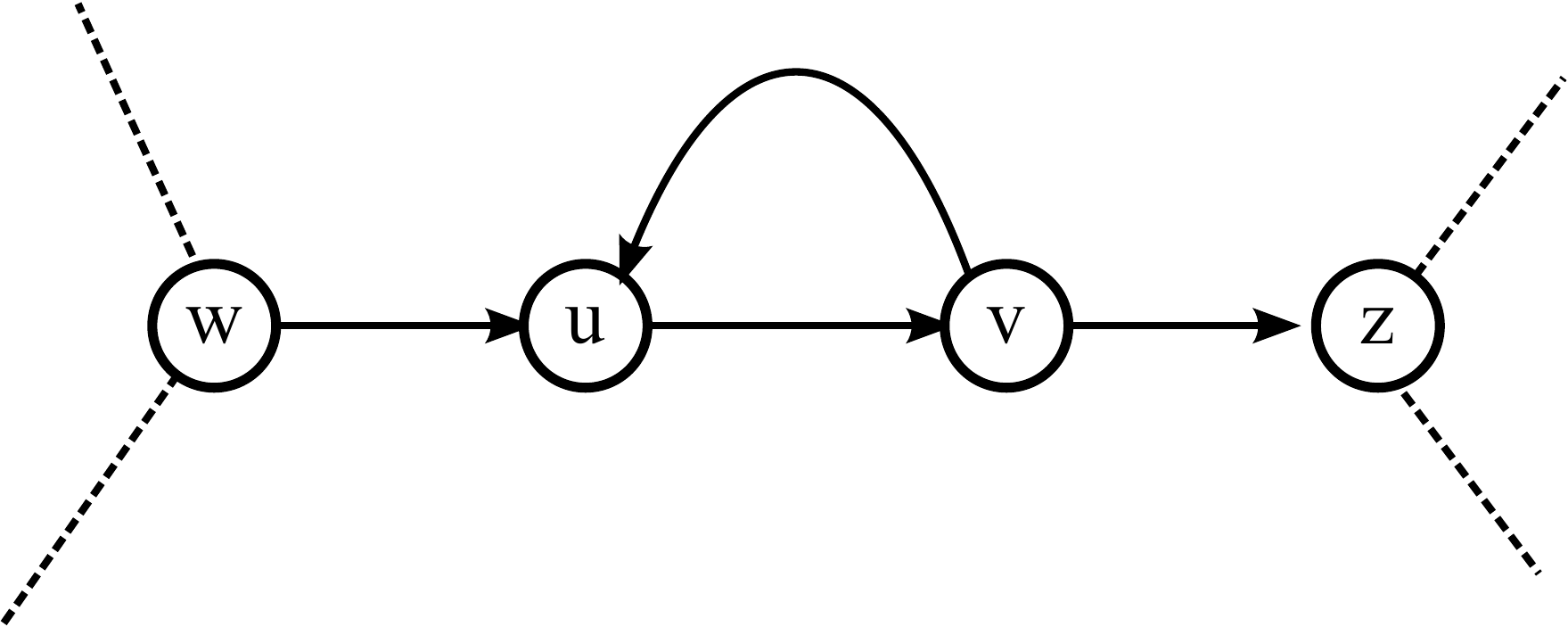}
	    \end{center}
	    \caption{The loop that must occur if $\#\PRECOLOR(\nu)$ is minimal.
		      The dotted edges do not assume a direction.}\label{FigLoop1}
	\end{figure}

	For $n\in \JJJ$, let $W^{(w,u,n)}$ be the word in $\LANGUAGE$ that represents the edge $w\to u$ in $\SPRGRAPH(n)$,
	    meaning $W^{(w,u,n)}$ begins with $w^{(n)}$, ends with
		$u^{(n)}$ and each
	      subword of length $n$ follows in order the simple path from $w^{(n)}$ to $u^{(n)}$ in $\RGRAPH(n)$.
	Define the prefix and suffix by $W^{(w,u,n)} =P^{(w,u,n)}u^{(n)}$ and
		      $W^{(w,u,n)} = w^{(n)}S^{(w,u,n)}$ respectively.
	Define similarly the path words $W$, suffixes $S$ and prefixes $P$ for the other three edges.
	We state the following lemma without proof, as it follows from the definition.
	A \emph{minimal return word} $W\in \LANGUAGE$ from $Y$ to $Z$ is a word so
	      that\footnote{For non-positive index values, we count from the right, i.e.
		$ W_{[i,j]} = W_{[|W| + i, |W| +j]}$
	    for $-|W|<i\leq j \leq 0$.}
	      $W_{[1,|Y|]} = Y$, $W_{[-|Z|+1,0]} = Z$ and $\#_Y(W) = \#_Z(W) = 1$.
	In other words, $W$ begins with $Y$, ends with $Z$ and no proper subword of $W$
	      begins with $Y$ and ends with $Z$.
	
	\begin{lem}\label{LemLoopWds}
	      Let $\JJJ$ satisfy Assumption \ref{ASSUMPTION2} with $u,v,w,z\in \BASEGRAPH$ as in Figure \ref{FigLoop1}.
	      For any $n\in \JJJ$, the each minimal return word in $\LANGUAGE$
		    from $W^{(w,u,n)}_{[1,n+1]}$ to $W^{(v,z,n)}_{[-n,0]}$
		      must be of the form
			      $$
					R_{u\to v,n}(m):= W^{(w,u,n)}\big[S^{(u,v,n)}S^{(v,u,n)}\big]^mS^{(u,v,n)}S^{(v,z,n)}
			      $$
				for some $m\geq 0$.
	\end{lem}

	\begin{defn}\label{DefLoopTimes}
	      For each $n\in \JJJ$ and loop $u\to v$ in $\BASEGRAPH$ as in Figure \ref{FigLoop1},
		  let
			$$
			      \LOOPTIMES_{u\to v}(n) = \{m\geq 0: R_{u\to v, n}(m)\in \LANGUAGE\}.
			$$
	      When the edge $u\to v$ is assumed, we will suppress this notation in $\LOOPTIMES(n)$.
	\end{defn}
	  Because we are considering a minimal aperiodic $\SPACE$,
	      for every loop $u\to v\in \BASEGRAPH$,
	      every $\LOOPTIMES_{u\to v}(n)$ is finite, although
	    the sizes may tend to infinity as $n\to \infty$.

	\subsection{Bispecial moves on loops}\label{SsecLoopsBS}
	    We consider $u,v,w,z$ for $\BASEGRAPH$ as in Figure \ref{FigLoop1}.
	    Fix $n\in \JJJ$ and let $n'' = |W^{(u,v,n)}|$.
	    The loop will remain in $\SPRGRAPH(n'')$ although bispecial moves may have occurred from $n$ to $n''$ elsewhere.
	    At step $n''$, the bispecial edge $u\to v$ now corresponds to bispecial word
		  $u^{(n'')} = v^{(n'')} = W^{(u,v,n'')}=W^{(u,v,n)}$.
	    The remaining words and prefixes are related in the following:
		  $$
			\begin{array}{rcl}
			      W^{(v,u,n'')} &=& P^{(u,v,n)}W^{(v,u,n)}S^{(u,v,n)},\\
			      P^{(v,u,n'')} &=& P^{(u,v,n)}P^{(v,u,n)},\\
			      S^{(v,u,n'')} &=& S^{(v,u,n)}S^{(u,v,n)},\mbox{ and }\\
			      P^{(u,v,n'')} = S^{(u,v,n'')} &=& \EMPTYWORD,
			\end{array}
		  $$
	    where $\EMPTYWORD$ is the empty word.
		The vertices $w^{(n'')}$
		and $z^{(n'')}$ will be appropriately defined depending on the vertex types and  bispecial moves
		on other edges that involve $w$ and $z$.

	   We will see potentially new local pictures in $\SPRGRAPH(n''+1)$ depending on the finite set
		    $\LOOPTIMES(n'') = \LOOPTIMES(n)$.
	  We now classify these possibilities.
	  First, for Figure \ref{FigLoop1} to occur (to have a loop at all), we must have $\max \LOOPTIMES(n'') \geq 1$;
		that is, the loop must be traversable at least once.
	    Let
		$$u' = W^{(u,v,u,n)}_{[1,n''+1]}, u'' = W^{(u,v,z,n)}_{[1,n''+1]},
		      v' = W^{(u,v,u,n)}_{[-n'',0]}\mbox{ and }v'' = W^{(w,u,v,n)}_{[-n'',0]}.$$
	    Here,
			$$W^{(u,v,u,n)} = W^{(u,v,n)}S^{(v,u,n)}$$
		  represents the path in $\SPRGRAPH(n)$ moving from $u$ to $v$ and then from $v$ to $u$,
	      with similar definitions for $W^{(u,v,z,n)}$ and $W^{(w,u,v,n)}$. 
	      Let $w'$ be the unique special word in $\LANGUAGE(n''+1)$ with special-avoiding path to $v''$ and similarly for $z'$. Then by definition the following paths must occur in $\RGRAPH(n''+1)$:
			  $$
				w' \rightsquigarrow v'',
				u'' \rightsquigarrow z'\mbox{ and }
				u' \rightsquigarrow v'.
			  $$
			  
			  The following cases arise at the bispecial word $u^{(n'')} = v^{(n'')}$ as we move from $\SPRGRAPH(n'')$ to $\SPRGRAPH(n''+1)$:
		\begin{enumerate}
		 \item If $\LOOPTIMES(n'') = \{1\}$, then the move is weak bispecial and the loop becomes an edge.
			  In Figure \ref{FigLoop2a}, the corresponding words $w'$ and $z'$ are the only relevant vertices
			      in $\SPRGRAPH(n''+1)$, as no other words are special.
		  \item If $\LOOPTIMES(n'') = \{0,1\}$, the move is RBS and
			    then there are now two edges of the form $v''\to u''$ as in Figure \ref{FigLoop2b}.
			    The words $u'$ and $v'$ are not special.
		  \item If $0\notin \LOOPTIMES(n'')$ and $\max \LOOPTIMES(n'') >1$, the move on $u\to v$ is RBS and
			    results in another loop about $u'$ and $v'$ as in Figure \ref{FigLoop2c}.
			    The words $u''$ and $v''$ are not special.
			  Note that
				$$
				      \LOOPTIMES_{u'\to v'}(n''+1) = \{m-1: m\in \LOOPTIMES(n'')\}.
				$$
		  \item If $0\in \LOOPTIMES(n'')$ and $\max \LOOPTIMES(n'') >1$,
			  then the loop $u' \to v' \to u'$ is still present, while the other vertices form edges
				  $w' \to v''$, $v'' \to u'$, $v'\to u''$, $v''\to u''$ and $u''\to z'$ as
			    indicated in Figure \ref{FigLoop2d}.
			  Similarly,
				$$
				      \LOOPTIMES_{u'\to v'}(n''+1) = \{m-1: m\in \LOOPTIMES(n''), m\neq 0\}.
				$$
		\end{enumerate}
      \begin{figure}[t]
	    \begin{center}
		\subfigure[Weak bispecial move.]{\label{FigLoop2a}
		      \includegraphics[width=.5\textwidth]{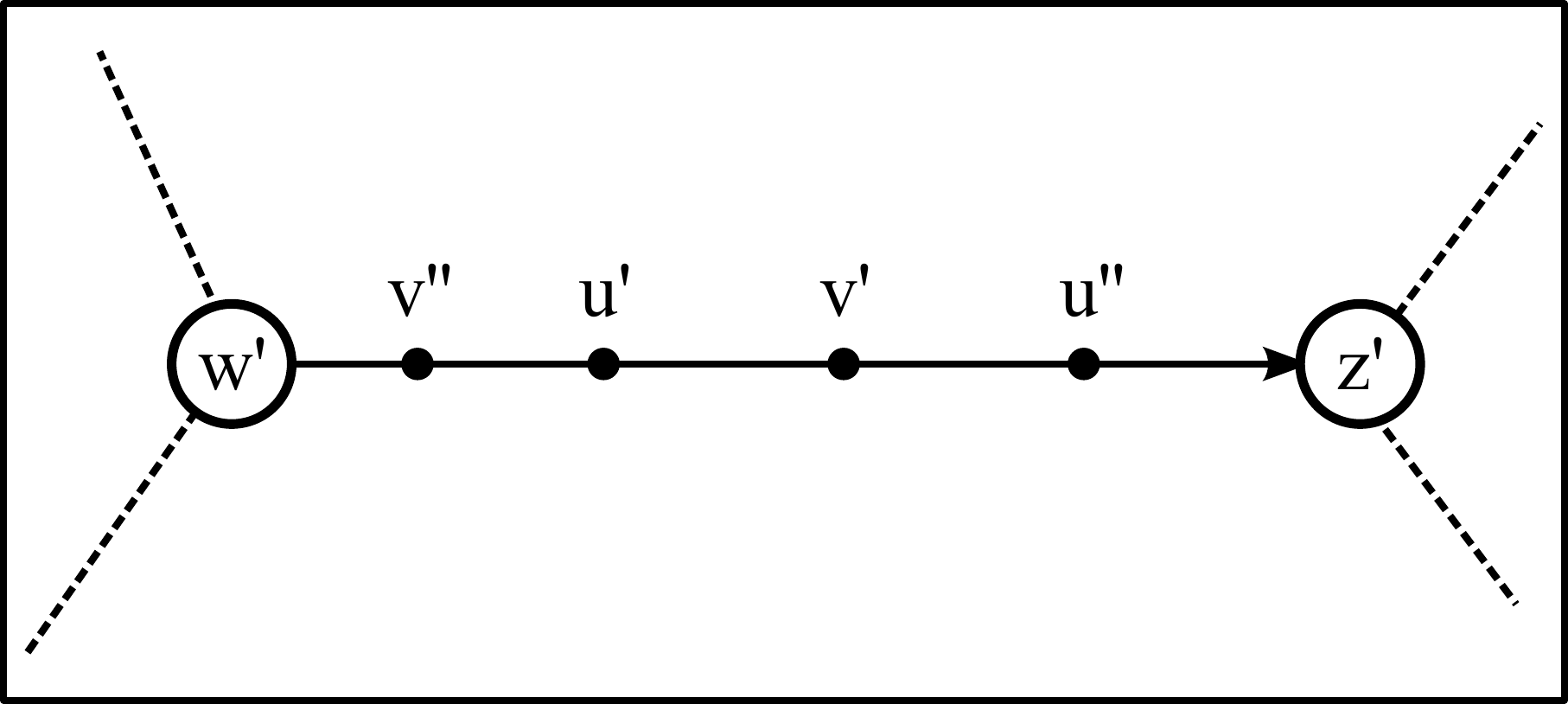}}
	    \end{center}
	    \begin{center}
		\subfigure[Regular bispecial move, removing the loop.]{\label{FigLoop2b}
		      \includegraphics[width=.5\textwidth]{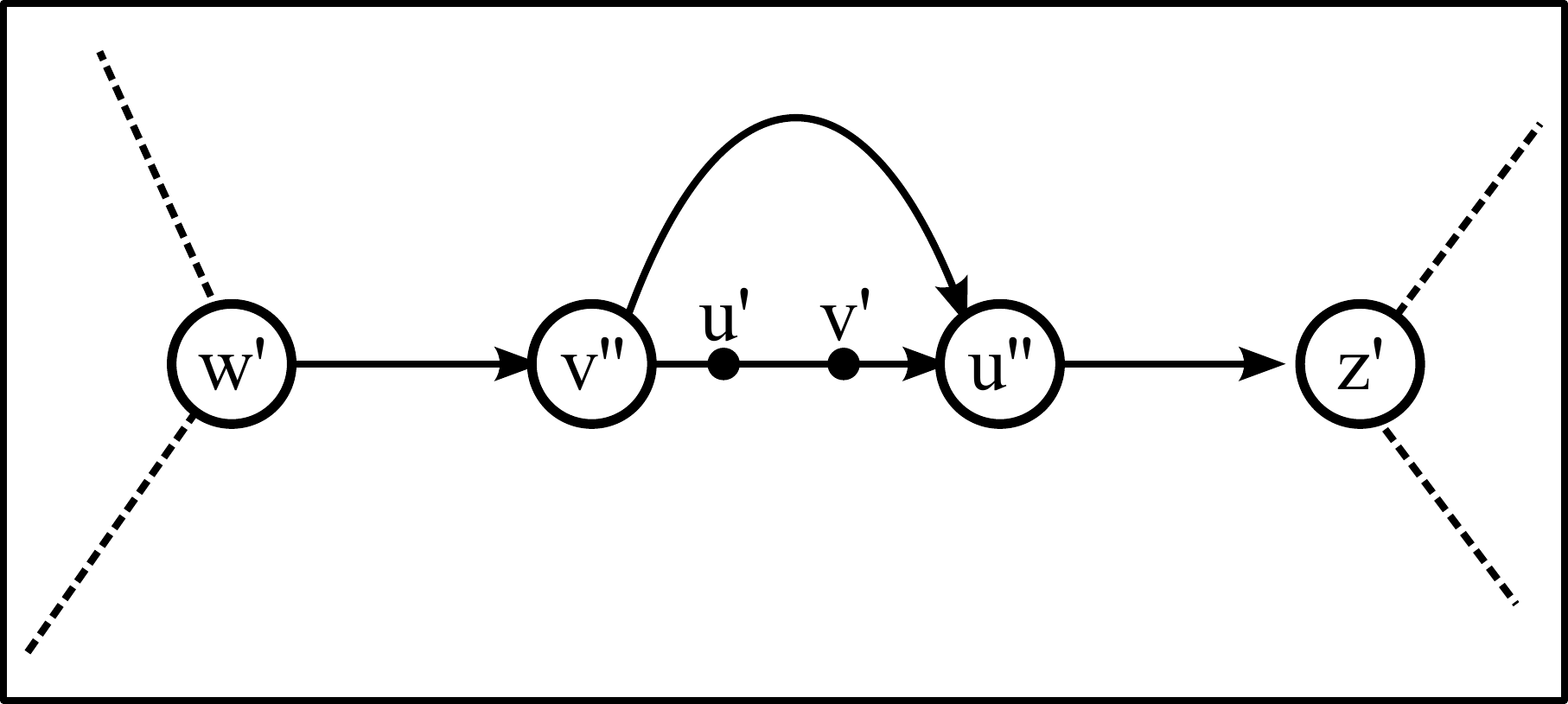}}
	    \end{center}
	    \begin{center}
		\subfigure[Regular bispecial move, maintaining the loop.]{\label{FigLoop2c}
		      \includegraphics[width=.5\textwidth]{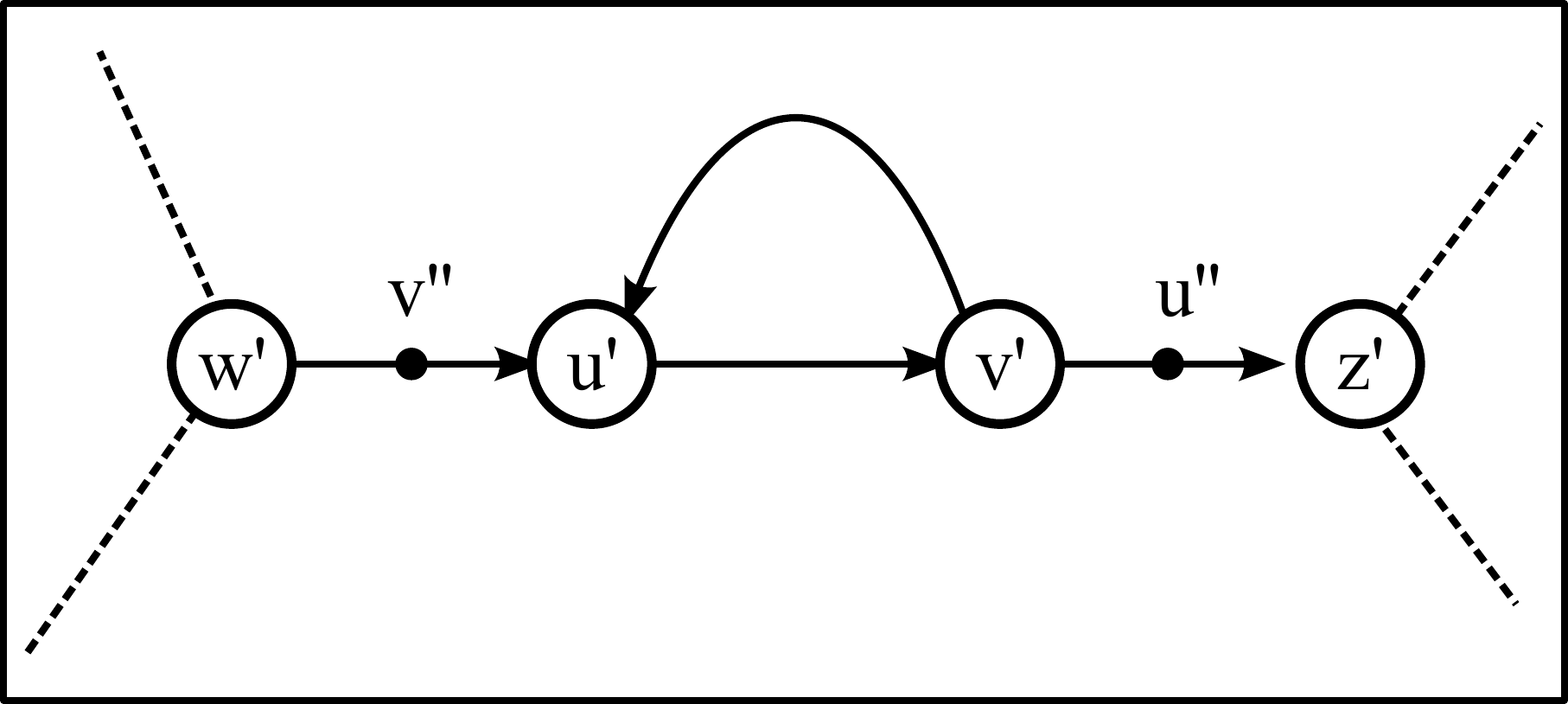}}
	    \end{center}
	    \begin{center}
		\subfigure[Strong bispecial move. The new shape is a loop atop a ``tower.'']{\label{FigLoop2d}
		      \includegraphics[width=.5\textwidth]{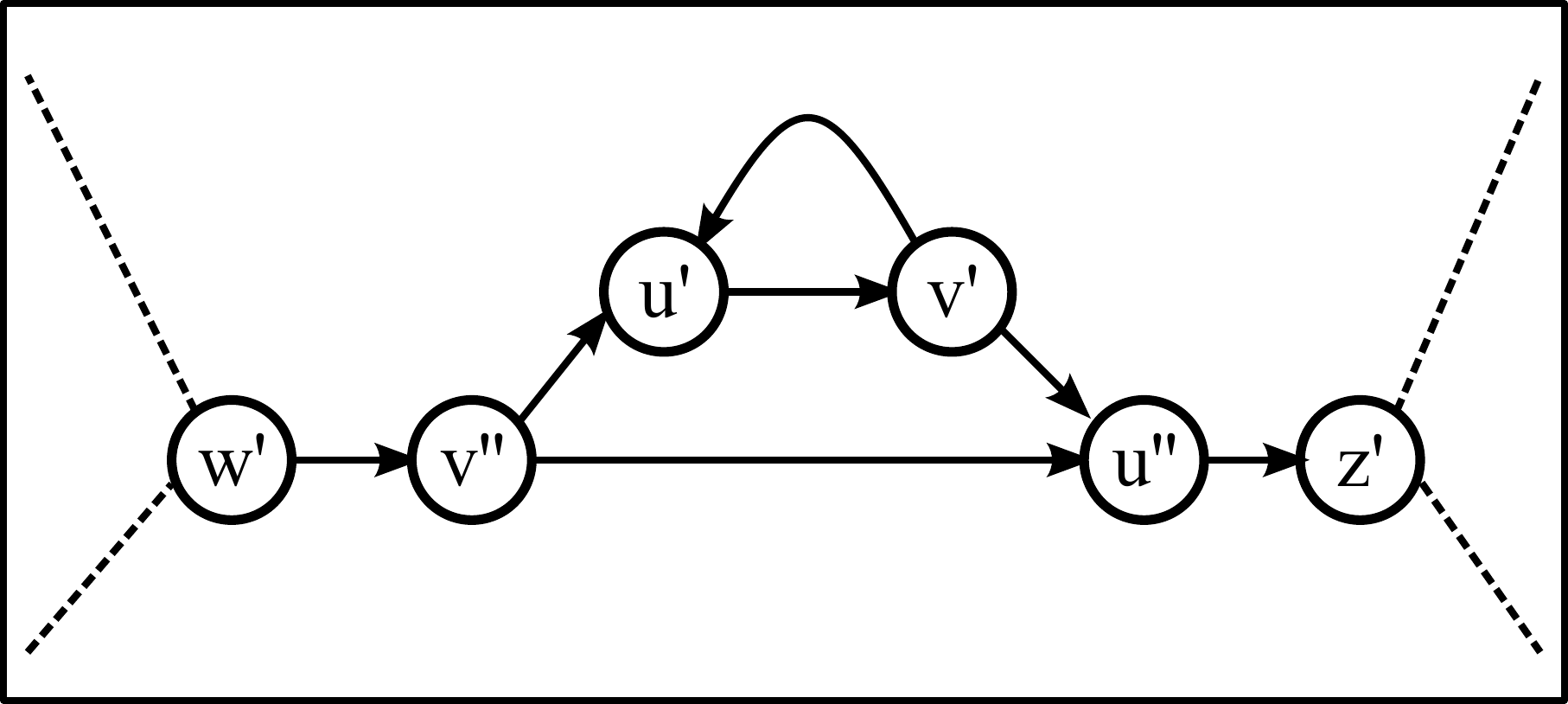}}
	    \end{center}
	    \caption{The types of bispecial moves on $u\to v$ from $\SPRGRAPH(n')$ in Figure \ref{FigLoop1} to $\SPRGRAPH(n'+1)$.
		      Circled nodes are vertices, while dot nodes are not.}\label{FigLoop2}
      \end{figure}

	\subsection{Coloring for loops}\label{SsecColoringLoops}
	  We will now discuss how the changes in the previous section affect colorings.
	  The first result says the following: if $\PRECOLOR(\nu)$ has only two elements, then the maximum number of windings
		about the corresponding loop must grow to infinity as $n\in \JJJ$ goes to infinity.
		
	  \begin{lem}
		If for $\nu\in \ERGODICsp$, $\PRECOLOR(\nu) = \{u,v\}$ with $u$ and $v$ as in Figure \ref{FigLoop1},
			  then $\lim_{\JJJ\ni n \to\infty} \max \LOOPTIMES(n) = \infty$.
	  \end{lem}

	  \begin{proof}
	      Suppose by contradiction that for some $n_0$ and $M$,
			      $$
			      		m\leq M \mbox{ for all } m\in \LOOPTIMES(n) \mbox{ and } n\geq n_0, n\in \JJJ.
			      $$
	      Fix $n \geq N_0$ and recall the generic $x$ for $\nu$.
	      Because the paths $u\to v \to u$ and $w\to u \to v$ in $\SPRGRAPH(n)$ have total length at most $Ln$,
			if $u^{(n)}$ occurs in position $p$ in $x$, then $w^{(n)}$ must occur in position $p'<p$ with $p-p' < (M+1)Ln$.
		By Lemma \ref{LemNbringDens}, this implies that $\DENS(\mu_w,\nu) \geq \frac{1}{3(M+2)} \DENS(\mu_u,\nu)>0$.
		But this is a contradiction because then $w \in \PRECOLOR(\nu)$.
		We may likewise show by contradiction that $z\in \PRECOLOR(\nu)$.
	  \end{proof}

	  The preceding proof yields the following natural converse.

	  \begin{cor}\label{CorBddLoopCount}
		If for $\nu\in \ERGODICsp$, $u,v\in \PRECOLOR(\nu)$ with $u$ and $v$ as in Figure \ref{FigLoop1} and
			  $\limsup_{\JJJ\ni n \to\infty} \max \LOOPTIMES(n) <\infty$, then
			$$
			      u,v,w,z \in \PRECOLOR(\nu),
			$$
		where $w$ and $z$ are the neighboring vertices to the loop.
	  \end{cor}

	  In the rest of the paper, when we transition from a special graph at stage $n''$ to $n''+1$,
		we would like ensure that $\SPRGRAPH(n'' + 1) \not\equiv \SPRGRAPH(n'')$.
	  Suppose $\SPRGRAPH(n'')$ contains at least one loop as in Figure \ref{FigLoop1}
		that will undergo a bispecial move.
	  If $\SPRGRAPH(n''+1) \equiv \SPRGRAPH(n'')$, then all such loops must have experienced
	      an RBS move as in Figure \ref{FigLoop2c}.
	  Thus, $\min \LOOPTIMES(n'') >0$ and $\max \LOOPTIMES(n'')>1$
		    for all such loops.
	  However, note that after the move, the new set $\LOOPTIMES(n''+1)$ satisfies
		  $$
		   \min \LOOPTIMES(n''+1) = \min \LOOPTIMES(n'') - 1
		   			\mbox{ and }
		   	\max \LOOPTIMES(n''+1) = \max \LOOPTIMES(n'') - 1.
		  $$
	  Therefore, we can for a fixed loop choose $n''' = n'' + m_0 a$ where $a$ is the total weight of the loop $u\to v \to u$; that is,
		   $a = |S^{(u,v,n'')}| + |S^{(v,u,n'')}|$,
	   and $m_0 = \min \LOOPTIMES(n'')$.
	    Then $\min \LOOPTIMES(n''') = 0$.
	    If $\LOOPTIMES(n''') = \{0\}$, then we no longer have a loop in $\SPRGRAPH(n''')$, as the bispecial move before $n'''$ was weak bispecial.
	    If $\LOOPTIMES(n''') = \{0,1\}$, then the move at $\BISPECIAL(n''')$ will be the regular bispecial move that
				removes the loop as in Figure \ref{FigLoop2b}.
	    Otherwise the move will be the strong bispecial move as in Figure \ref{FigLoop2d}.
	    For all $n''\leq \tilde{n} \leq n'''$, the loop will persist.
	    
	    For the next lemma, recall that $L = K+1$.

	    \begin{lem}
		  Let $n''$ be so that $\SPRGRAPH(n'')$ has loop about $u^{(n'')}=v^{(n'')}$ as in Figure \ref{FigLoop1},
		    meaning in particular that $u^{(n'')} = v^{(n'')}$ is bispecial.
		   If $n''' = n'' + a b$,
		      where $a$ is the length of the loop $u^{(n'')} \to v^{(n'')} \to u^{(n'')}$ and $b \leq \min \LOOPTIMES(n'')$,
		  then
			$$
			      \DENS(u^{(n'')},x) \geq \frac{1}{12(K+1)}\DENS(u^{(n''')},x)
			      	\mbox{ and }
			      \DENS(u^{(n''')},x) \geq \frac{1}{27}\DENS(u^{(n'')},x)
			$$
		  for any $x$.
	    \end{lem}

	    \begin{proof}
		  Because $b \leq \min \LOOPTIMES(n'')$, each occurrence of $u^{(n'')}$ is contained in an occurrence of $u^{(n''')}$.
		  So we apply Corollary \ref{CorSbwdLoopDensSmall} with $\alpha = K+1$.
	    \end{proof}

	    For a loop in $\SPRGRAPH(n)$, let $n'' \geq n$ be the minimum value such that $u^{(n'')} = v^{(n'')}$.
	    Furthermore, let $n''' = n'' + ab$ as in the lemma with $b$ taken to be the maximum such value
	      so that the loop remains in $\SPRGRAPH(n''')$; that is, $b$ is the minimum of $\min\LOOPTIMES(n'')$
		  and $\max\LOOPTIMES(n'')-1$.
	    If there are multiple loops, then let $n'''$ be the minimum of all such values.
	    We now choose $\JJJ''$ from $n'''$ for each $n\in \JJJ$ so that
		   $\SPRGRAPH(n''') \equiv \BASEGRAPH''$, $\SPRGRAPH(n'''+1) \equiv \BASEGRAPH'''$,
	      and $\JJJ''$ with $\BASEGRAPH''$ satisfies Assumption \ref{ASSUMPTION2}.
	    We then reduce $\JJJ$ so that $\JJJ = \{n_1,n_2,\dots\}$ and $\JJJ'' = \{n_1''',n_2''',\dots\}$
		    satisfy
		    $$ n_1<n_1'''<n_2<n_2'''<n_3<n_3'''<\dots.$$
	    Note that $\BASEGRAPH''\not\equiv \BASEGRAPH '''$.
	    The next result tells us that colors of loops in $\BASEGRAPH$ persist to $\BASEGRAPH''$.
	    
	    \begin{cor}\label{CorAccelLoops}
		With $\BASEGRAPH$ and $\BASEGRAPH''$ as above,
			for each $u$ and $v$ associated to a loop
			  $$
			      \COLOR(u) = \COLOR(v) = \COLOR''(u)= \COLOR''(v)
			  $$
		where $\COLOR$ is the coloring relation for $\BASEGRAPH$ and $\COLOR''$
			is the coloring relation for $\BASEGRAPH''$.
	    \end{cor}

	  \section{Proof of Main Theorem}\label{SecMainProof}

	      We will first prove Theorem \ref{ThmMain} under the  binary extension condition.

	      \begin{prop}\label{PropMainBin}
		  If minimal shift $\SPACE$ on finite $\ALPHABET$ satisfies equation \eqref{EqFixedKComp} with $K \geq 4$ and its
			language satisfies the binary extension condition,
		  then $\#\ERGODICsp \leq K-2$.
	      \end{prop}
	      	  
	  	\subsection{Binary Extension Condition: Special Cases}

		Given $K$, consider a special (unweighted) graph $\BASEGRAPHTOWER$ defined by the conditions:
			\begin{itemize}
				\item there are exactly $K-1$ bispecial edges,
				\item all $K-1$ bispecial edges are in loops as in Figure \ref{FigLoop1},
				\item there is one loop $u' \to v'$ and vertices $u'',v''$ so that
						these four vertices form a ``tower'' resulting from a SBS move on loop $u'\to v'$
						as in Figure \ref{FigLoop2d}.
			\end{itemize}
		The graph $\BASEGRAPHTOWER$ is represented in Figure \ref{FigOneTower}.
	      \begin{figure}[t]
		  \begin{center}
			\includegraphics[width = .7\textwidth]{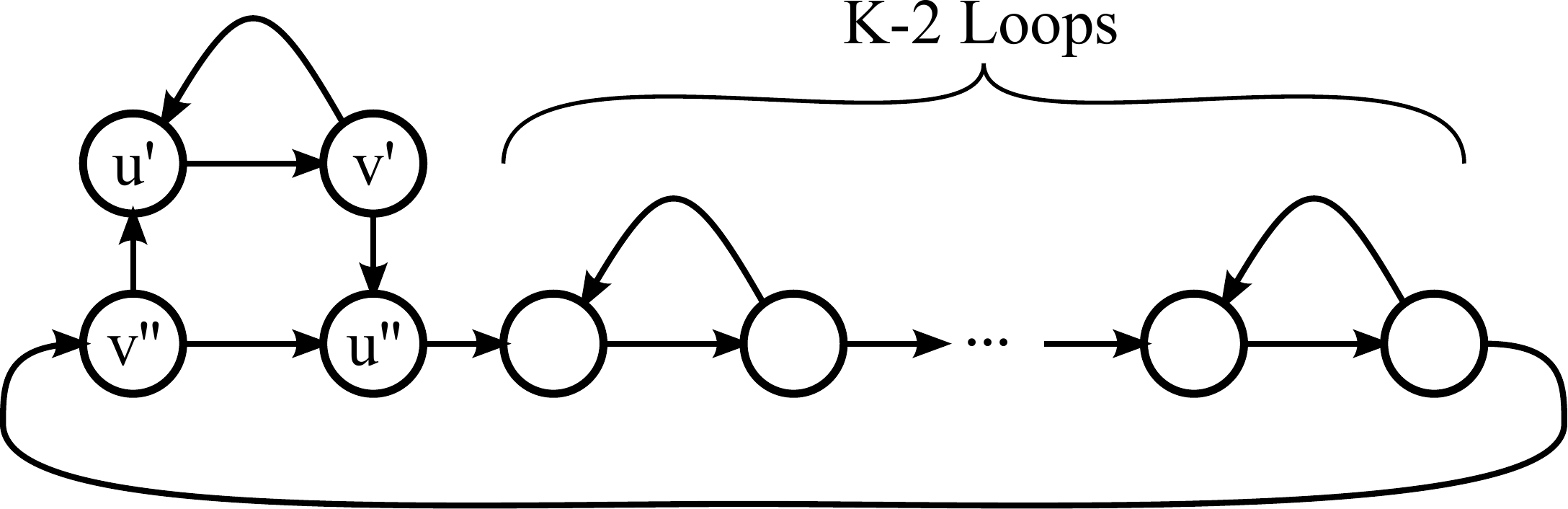}
		  \end{center}
		  \caption{The graph $\BASEGRAPHTOWER$.}\label{FigOneTower}
	      \end{figure}
		In the proof of the next result,
		    we will require that the base vertices $u'',v''\in \BASEGRAPHTOWER$
		    connect to distinct loops outside of the tower $u',v',u'',v''$.
		This occurs only for $K\geq 4$.
		\begin{lem}\label{LemMainOneTower}
			If $\SPACE$ satisfies equation \eqref{EqFixedKComp}
				with $\BASEGRAPH$ and $K\geq 4$ from Assumption \ref{ASSUMPTION2}
				where $\BASEGRAPH \equiv \BASEGRAPHTOWER$, then $\#\ERGODICsp \leq K-2$.
		\end{lem}
		
		\begin{proof}
		Note that Corollary \ref{Cor2ECPropogates} implies that the language satisfies the
			    binary extension condition and so every special Rauzy graph $\SPRGRAPH(n)$ for sufficiently
			    large $n$ has exactly $2K$ vertices.
			Let $E = \#\ERGODICsp$.
			For $\JJJ$, accelerate to $\JJJ''$ with $\BASEGRAPH''$ as discussed
				before Corollary \ref{CorAccelLoops} on the $K-1$ loops.
			Because there are no bispecial edges except these loops, then necessarily
				$\BASEGRAPH'' \equiv \BASEGRAPH \equiv \BASEGRAPHTOWER$.
			If either of the base vertices $u''$ or $v''$ of the tower are colored by some
				$\nu_0\in\ERGODICsp$ then we may see by Proposition \ref{PropColorCycles}
				that there are at least two loops in the graph that are colored by the same measure.
			Because $\#\PRECOLOR(\nu_0)\geq 5$, $E \leq K-2$ by Corollary \ref{CorDamronsStar}.
			We must have then that $\COLOR(u'') = \COLOR(v'') = \NIL$ and if
				any loop is also uncolored, we have $\#\PRECOLOR(\NIL) \geq 4$ and again
			      Corollary \ref{CorDamronsStar} yields $E\leq K-2$.
			
			So assume that each loop is colored and the extra two vertices $u''$ and $v''$ are not.
			Consider the next move from $\BASEGRAPH'' \equiv \BASEGRAPHTOWER$ to $\BASEGRAPH'''$
				also mentioned before Corollary \ref{CorAccelLoops}.
			At least one loop must undergo an RBS or WBS change, as at least one loop
				must undergo one of the three bispecial changes in Figure \ref{FigLoop2};
					that is, not the move in Figure \ref{FigLoop2c} that preserves the loop.
			If a loop undergoes a SBS change then another loop must undergo a WBS change to preserve
			    the total number of vertices.
			
			Fix a loop $u\to v$ that undergoes an RBS or WBS change
					and write $\nu_1$ for he measure with $u,v\in\PRECOLOR(\nu_1)$.
				For this loop, $\LOOPTIMES(n''') = \{0,1\}$ or $\{1\}$ for each $n'''\in \JJJ''$.
			By Corollary \ref{CorBddLoopCount}, the vertices in $\BASEGRAPH''$ adjacent to
				the loop $u\to v$ must share its color.
			If this loop is either $u'\to v'$ (the top of the ``tower'') or adjacent to $u''$ or $v''$ (the base vertices of the
				``tower''),
				we have a contradiction as either $u''$ or $v''$ must belong to $\PRECOLOR(\nu_1)$.
			Otherwise, this loop then shares its color with its two neighboring loops, and so
				$\#\PRECOLOR(\nu_1) \geq 6$ and so $E\leq K-2$ by Corollary \ref{CorDamronsStar}.
		\end{proof}

		\begin{lem}\label{LemMainManyLoops}
			Suppose $\SPACE$ satisfies equation \eqref{EqFixedKComp} with $\BASEGRAPH$
				and $K\geq 4$ from
				Assumption \ref{ASSUMPTION2}
			If $\BASEGRAPH$ contains at least $K-2$ consecutive loops that are colored,
				meaning that there are vertices $u^{(k)},v^{(k)}\in \BASEGRAPH$,
					$1\leq k \leq K-2$ so that
					\begin{enumerate}\renewcommand{\theenumi}{\roman{enumi}}
						\item for each $1\leq k \leq K-2$, $u^{(k)} \to v^{(k)}$
							forms a loop as in Figure \ref{FigLoop1},
						\item for each $1 \leq k < K-2$, the edge $v^{(k)}\to u^{(k+1)}$ exists in $\BASEGRAPH$ and connects the $k^{th}$ loop to the $(k+1)^{st}$ and
						\item for each $1\leq k \leq K-2$, $\COLOR(u^{(k)}) = \COLOR(v^{(k)}) \neq \NIL$,
					\end{enumerate}
				then $\#\ERGODICsp\leq K-2$.
		\end{lem}
		
		\begin{proof}
			Let $E = \#\ERGODICsp$.
			Recalling Corollary \ref{CorDamronsStar}, we will show by cases that either:
			      there is a measure $\nu_0$ so that at least 5 vertices are colored by $\nu_0$ or 
				there are at least 3 uncolored vertices.

			We will consider a number $\KKK\geq K-2$ of consecutive colored loops.
			If $\KKK = K$, then $\BASEGRAPH$ is just a cycle of $K$ colored loops.
			In this case, consider
				$\JJJ'',\BASEGRAPH'', \JJJ''', \BASEGRAPH'''$
			from the discussion before Corollary \ref{CorAccelLoops}.
			Just as in the previous proof,
				$\BASEGRAPH''\equiv \BASEGRAPH$ and
				all loops are still colored by relation $\COLOR''= \COLOR$ on $\BASEGRAPH''$.
			Also, at least one loop will undergo either an RBS or WBS change
				from $\BASEGRAPH''$ to $\BASEGRAPH'''$ and so
				by Corollary \ref{CorBddLoopCount},
			its color will be shared by the neighboring loops.
			Therefore if $\nu_0$ is the coloring for that loop, then
				$\#\PRECOLOR(\nu_0) \geq 6$.
				
			If $\KKK = K-1$, then either:
				\begin{itemize}
					\item[(A)] there are exactly $K-1$ loops and the remaining two
									vertices $u^{(K)},v^{(K)}$
									are connected by the edges $v^{(K-1)} \to v^{(K)}$,
										$u^{(K)}\to u^{(1)}$
									and two edges $v^{(K)}\to u^{(K)}$, or
					\item[(B)] there are $K$ loops in a cycle, but one is not colored.
							Call the vertices of this loop $u^{(K)} \to v^{(K)}$.
				\end{itemize}
				
			If (A) holds, construct $\JJJ'',\BASEGRAPH'',\JJJ''',\BASEGRAPH'''$ as before.
				Then $\BASEGRAPH'' = \BASEGRAPH$ (the only bispecial edges exist in the $K-1$ loops)
				    and
				from $\BASEGRAPH''$ to $\BASEGRAPH'''$ there must exist a loop
				colored by some $\nu_0$ that undergoes either an RBS or WBS change.
			By Corollary \ref{CorBddLoopCount}, the vertices adjacent to this loop
					must also be colored by $\nu_0$.
					Such an adjacent vertex is either an element of another loop or of
					      the set $\{u^{(K)}, v^{(K)}\}$.
					In either case, the color $\nu_0$ on these vertices implies that at least
					      two vertices on each side of the original loop are colored by
					      $\nu_0$ as well.
			Therefore again $\#\PRECOLOR''(\nu_0) \geq 6$.
			
			If (B) holds,  construct $\JJJ'',\BASEGRAPH'',\JJJ'',\BASEGRAPH'''$ as before
				but focus on all $K$ loops.
			If a colored loop in $\BASEGRAPH'' \equiv \BASEGRAPH$ undergoes an RBS or WBS change,
				then just as in the previous case $\#\PRECOLOR''(\nu_0)\geq 6$ for 
				$\nu_0$ coloring that loop.
			If not, then either
				\begin{itemize}
					\item[(B.1)] no colored loop changes from $\BASEGRAPH''$ to $\BASEGRAPH'''$ and
								so the loop $u^{(K)} \to v^{(K)}$ undergoes an RBS change, or
					\item[(B.2)] exactly one colored loop undergoes an SBS change and the loop
								$u^{(K)}\to v^{(K)}$ undergoes a WBS change.
				\end{itemize}
			If (B.1) occurs, then $\BASEGRAPH'''$ is of the form in (A).
				We apply that
					argument to $\JJJ'''$ and $\BASEGRAPH'''$ to
					conclude that $E\leq K-2$.				
			If (B.2) occurs, then $\BASEGRAPH''' \equiv \BASEGRAPHTOWER$ and
				    so $E\leq K-2$ by Lemma \ref{LemMainOneTower}.
						
			For the last case, suppose $\KKK = K-2$.
			Find $\JJJ'',\BASEGRAPH'',\JJJ''',\BASEGRAPH'''$ by considering
				only the $K-2$ loops.
			That is, the $K-2$ loops persist to $\BASEGRAPH''$ but a bispecial
			  change will affect at least one of them from $\BASEGRAPH''$ to $\BASEGRAPH'''$.
			Again, if a colored loop experiences an RBS or WBS change
					from $\BASEGRAPH''$ to $\BASEGRAPH'''$,
					then $\#\ERGODICsp\leq K-2$.
			If not, then either:
					\begin{itemize}
						\item[(C)] two colored loops undergo SBS moves, or
						\item[(D)] one colored loop undergoes an SBS move.
					\end{itemize}
			Here, we have used that three SBS moves are impossible by a counting argument. In either case, the remaining colored loops do not change.
			If (C) occurs, then the remaining four vertices in $\BASEGRAPH''$
				not in a colored loop will form two bispecial edges
				that will undergo WBS moves from $\BASEGRAPH''$ to $\BASEGRAPH'''$.
			Therefore, $\BASEGRAPH'''$ will have exactly $K-2$ loops and two ``towers''
				from SBS moves, each tower composed of a loop and two base vertices.
			If any of the four base vertices in $\BASEGRAPH'''$ is colored by some measure $\nu_0$,
			      then by Proposition \ref{PropColorCycles} it follows that $\#\PRECOLOR'''(\nu_0) \geq 5$. This is because if the base vertices of a tower are colored, then all four vertices in the tower have the same color.
			If none of the base vertices are colored, then $\#\PRECOLOR'''(\NIL) \geq 4$.
						
			If (D) occurs, then two of the four remaining vertices in $\BASEGRAPH''$
				will belong to a bispecial edge that will undergo a WBS move.
			Therefore, $\BASEGRAPH'''$ will contain a loop tower and at least $K-3$ loops,
				all inheriting colors from the $K-2$ loops in $\BASEGRAPH''$ by
				Lemma \ref{LemColorMoves}.
			The two extra vertices are either both colored or both not
				by considering the graph structure.
			If the two vertices are not colored or share a color with
				one of the loops, then by excluding these two vertices as well
				as the four for the tower we have that $2(E-1) \leq 2K - 6$ and again
				$E\leq K-2$.
			If the two vertices are colored by a different measure, then they must form a
				loop as in Figure \ref{FigLoop1}.
			In this case $\BASEGRAPH''' = \BASEGRAPHTOWER$ and by Lemma \ref{LemMainOneTower},
				$E\leq K-2$.
			We have concluded the proof as all cases have been exhausted.
		\end{proof}
		
		\subsection{Binary Extension Condition: Main Proof}
		
		\begin{proof}[Proof of Proposition \ref{PropMainBin}]
				
				Construct $\JJJ$ and $\BASEGRAPH$ that satisfy Assumption \ref{ASSUMPTION2} for
					$\SPACE$.
				Suppose first by contradiction that $E= K$, where $E = \#\ERGODICsp$.
				Then $\BASEGRAPH$ must be $K$ colored loops all in a cycle.
				However, by Lemma \ref{LemMainManyLoops} it must be that $E\leq K-2$ and we
					have contradicted our assumption.
					
				Now suppose by contradiction that $E = K-1$.
				Because
					$$
						2(K-1) \leq \sum_{\nu} \#\PRECOLOR(\nu) \leq 2K
					$$
				we must have
					$$
						\#\PRECOLOR(\nu_1) = \dots = \#\PRECOLOR(\nu_{K-3}) = 2 \leq 
						\#\PRECOLOR(\nu_{K-2}) \leq \#\PRECOLOR(\nu_{K-1})
					$$
				for some choice of ordering $\nu_1,\dots,\nu_{K-1}\in \ERGODICsp$.
				By focusing on the $K-3$ colored loops,
					choose $\JJJ'',\BASEGRAPH'',\JJJ''',\BASEGRAPH'''$
					as discussed before Corollary \ref{CorAccelLoops}.
				First suppose that one of the colored loops undergoes a WBS or RBS move from $\BASEGRAPH''$
					to $\BASEGRAPH'''$, then for its measure $\nu_{k_0}$ we have
					$\#\PRECOLOR''(\nu_{k_0})\geq 4$ by Corollary \ref{CorBddLoopCount}.
				Note that if $\#\PRECOLOR(\nu_{k_0}) \geq 5$, then $E \leq K-2$ by Corollary \ref{CorDamronsStar},
				      a contradiction.
					
				If we name this loop $u\to v$ with adjacent vertices $w$
					and $z$ as in Figure \ref{FigLoop1}, then $\#\PRECOLOR''(\nu_{k_0}) = 4$ if and only if
					$z\to w$ forms an ``outer'' loop that nests loop $u\to v$ as in
					Figure \ref{FigNested}.
					\begin{figure}[t]
				      \begin{center}
					    \includegraphics[width = .6\textwidth]{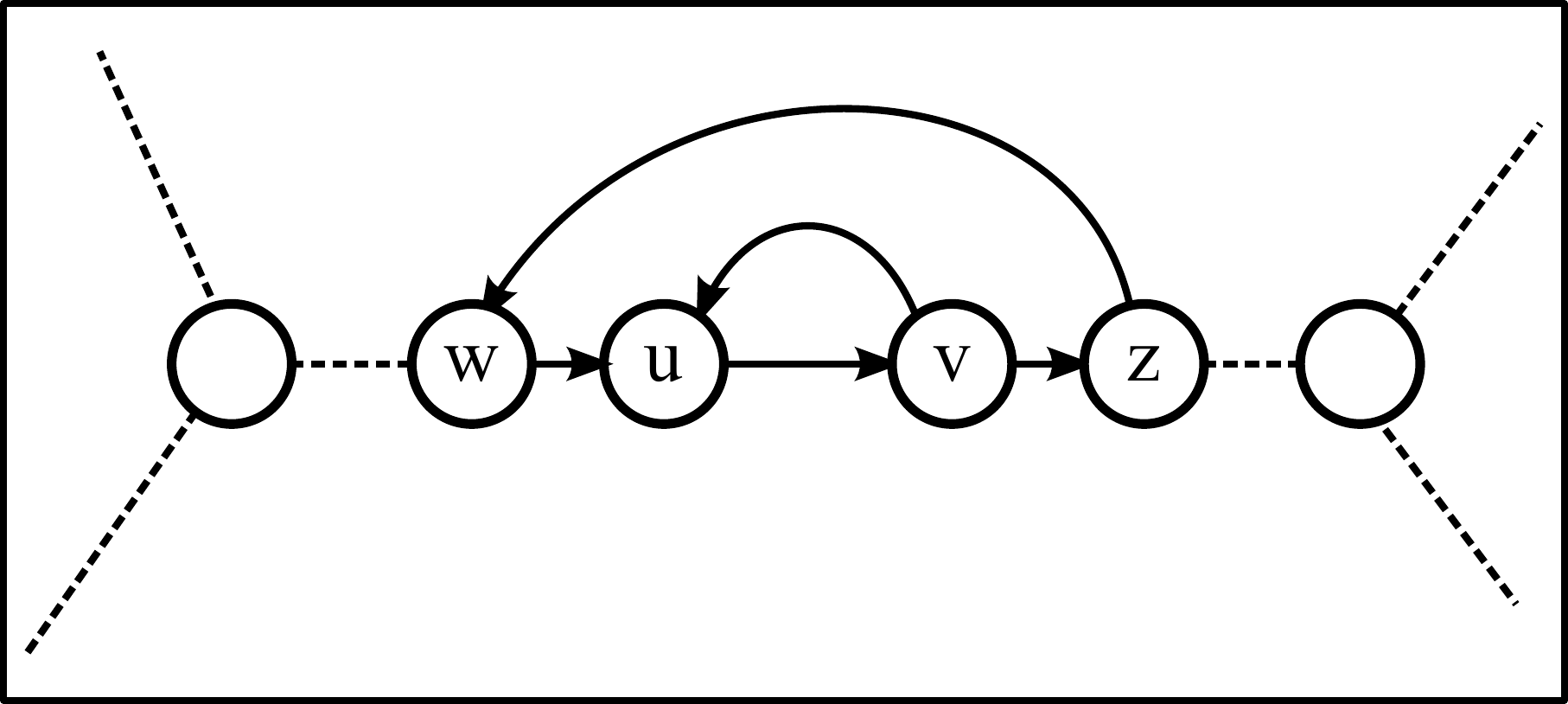}
				      \end{center}
				      \caption{Loop $u\to v$ is ``nested'' in larger
				      	loop $w\to u \to v \to z$. Note that the dotted edges incident to $w$ and $z$ do not have an orientation. This is because either $w$ is left special and $z$ is right special, or $w$ is right special and $z$ is left special.}\label{FigNested}
					\end{figure}
				Note that exactly one vertex $w$ or $z$ is left special and the other is right.
				By counting the remaining measures and using the assumption $E= K-1$,
					\begin{equation}\label{EqMain1}
						2(K-2) = 2(E-1) \leq \sum_{\nu\neq \nu_{k_0}} \#\PRECOLOR''(\nu) \leq 2K-4,
					\end{equation}
				so the remaining measures all color distinct loops in $\BASEGRAPH''$.
				In this case $\BASEGRAPH''$ is the nested loop connected in a cycle to the remaining
					(and consecutive) $K-2$
					colored loops.
				Again we have a contradiction that $E\leq K-2$ by Lemma \ref{LemMainManyLoops}.

				Now suppose that no colored loop in $\BASEGRAPH''$ will undergo an RBS or WBS change
					to $\BASEGRAPH'''$.
				Then at least one of the $K-3$ loops will undergo an SBS change while the remaining
					loops do not change.
				Each SBS change will result in a tower that must share its color
					with its corresponding loop in $\BASEGRAPH$.
				If $M$ such towers are created, then
					$$
						2(K-1-M) = 2(E-M) \leq 2K - 4M \Rightarrow M\leq 1.
					$$
				Therefore exactly one SBS tower will be created in $\BASEGRAPH'''$
				and the remaining loops will persist from $\BASEGRAPH$ to this $\BASEGRAPH'''$.
				If $\nu_{k_0}$ is the measure related this new tower, then the inequality
					\eqref{EqMain1} holds but by summing for $\PRECOLOR'''$ instead.
				We again conclude that $\BASEGRAPH''' \equiv \BASEGRAPHTOWER$ and so we
					reach our contradiction as $E\leq K-2$.
		\end{proof}

      \subsection{General Languages: Main Proof}

	  Now suppose $\SPACE$ has eventually constant complexity growth $K$ as in equation \eqref{EqFixedKComp}
		but does not satisfy the binary extension condition.
	  For $\mathfrak{s}\in \{\ell,r\}$, let
		  $$
		      \Psi_\mathfrak{s} = \lim_{n\to\infty} \psi_\mathfrak{s}(n)
		  $$
	  where $\psi_\mathfrak{s}(n) = \max\{\#\EXTENDs(w):w\in \LANGUAGE(n)\}$ from Lemma \ref{LemDecPsi}.
	  Because the language $\LANGUAGE$ does not satisfy the binary extension condition,
	      $\Psi_\mathfrak{s} >2$ for some $\mathfrak{s} \in \{\ell,r\}$.
	  By equation \eqref{EqPsiSpecBd} and Corollary \ref{CorKnBd},
	      if $\Psi_\mathfrak{s} \geq 4$ for some $\mathfrak{s}$,
	      then $\#\ERGODICsp \leq K-2$.
	  Items (I)-(III) of the following lemma are consequences of this fact;
		  the remaining statement follows from equation~\eqref{EqExtToGrowth}.
	  \begin{lem}\label{LemMainGenCases}
	  If $\SPACE$ satisfies the conditions above with $\JJJ,\BASEGRAPH$ satisfying Assumption \ref{ASSUMPTION2}
	      and $K_\mathfrak{s} \geq K-1$ for each $\mathfrak{s} \in \{\ell,r\}$, then
	      exactly one of the following must hold:
		\begin{enumerate}\renewcommand{\theenumi}{\Roman{enumi}}
		 \item $\Psi_\ell = 2$ and $\Psi_r = 3$,
		  \item $\Psi_\ell = 3$ and $\Psi_r = 2$, or
		  \item $\Psi_\ell = \Psi_r = 3$.
		\end{enumerate}
	  Furthermore, if $\Psi_\mathfrak{s} = 3$ then for all large $n$ there exists a unique $\tilde{w}\in \LANGUAGEs(n)$
		so that $\#\EXTENDs(\tilde{w}) =3$ and
		for all $w\in \LANGUAGEs(n)$,
		    $w\neq \tilde{w}$,
		$\#\EXTENDs(w) = 2$.
	  \end{lem}

      \begin{proof}[Proof of Theorem \ref{ThmMain}]

	    If $\SPACE$ has language $\LANGUAGE$ that satisfies the binary extension condition,
		    then Proposition \ref{PropMainBin} implies that $E \leq K-2$, where $E = \#\ERGODICsp$.
	    Otherwise, one of the cases in Lemma \ref{LemMainGenCases} holds.
	    Note in any of these cases, $K_\mathfrak{s} = K-1$ for some $\mathfrak{s} \in \{\ell,r\}$,
		    so $E = K$ is not possible by Corollary \ref{CorKnBd}.
	    We will therefore assume for a contradiction that $E = K-1$, and argue by cases.
	    Construct $\JJJ$, $\BASEGRAPH$, $K_\ell$ and $K_r$ that satisfy Assumption \ref{ASSUMPTION2}.
	    For each $\mathfrak{s}\in \{\ell,r\}$,
		let $\PRECOLOR_\mathfrak{s}$ denote the elements of
		$\PRECOLOR$ that are $\mathfrak{s}$-special.
 
	    If case (I) holds, then $K_r = K-1$ and $K_\ell = K$.
	    So $\#\PRECOLOR_r(\nu) = 1$ for all $\nu \in \ERGODICsp$ and
		  $\PRECOLOR_r(\NIL) = \emptyset$.
	    Furthermore, either $\#\PRECOLOR_\ell(\nu) = 1$ for all $K-1$ measures $\nu$ and $\#\PRECOLOR_\ell(\NIL)=1$,
		  or $\#\PRECOLOR_\ell(\nu)=1$ for $K-2$ measures $\nu$ and $\#\PRECOLOR_\ell(\nu)=2$ for one measure $\nu$.
	    In either case, if $\nu_0$ is the measure such that $\PRECOLOR_r(\nu_0)$ contains the vertex of out-degree three,
		  there must be at least $K-3$ measures $\nu \neq \nu_0$ such that $\#\PRECOLOR_r(\nu) = \#\PRECOLOR_\ell(\nu)=1$.

	    Construct $\JJJ''$, $\BASEGRAPH''$, $\JJJ'''$, $\BASEGRAPH'''$ as before Corollary \ref{CorAccelLoops}
		so that the $K-3$ binary loops are preserved from $\BASEGRAPH$ to $\BASEGRAPH''$ and at least one changes
		from $\BASEGRAPH''$ to $\BASEGRAPH'''$.
	    Let $\COLOR''$ and $\COLOR'''$ be the coloring functions on $\BASEGRAPH''$ and $\BASEGRAPH'''$ respectively.
	    By Corollary \ref{CorAccelLoops}, $\COLOR'' = \COLOR$ on each binary loop.
	    We claim that at least one binary loop will undergo an RBS or WBS change from $\BASEGRAPH''$ to $\BASEGRAPH'''$.
	    Otherwise, at least one binary loop must undergo an SBS change from $\BASEGRAPH''$ to $\BASEGRAPH'''$.
	    However, then the created tower has two right special vertices of the same color,
		    and this contradicts $\#\PRECOLOR_r(\nu)=1$ for all $\nu \in \ERGODICsp$.
	    Therefore a binary loop $u\to v$ undergoes an RBS or WBS
		    change and must share its color with the next vertex $a$ on the path from $v$ leading away from $u$.
	    Since $\#\PRECOLOR_r(\nu)=1$ for all $\nu \in \ERGODICsp$,
		  $a$ cannot be right special, so it is left special.
	    Following the path from $a$,
		  we must eventually hit a right special vertex,
		  and each left special vertex shares the same color as the loop $u\to v$. So the first right special vertex we hit must be $v$, and this contradicts minimality.

	    The case (II) may be handled as above by interchanging the roles of ``left special'' and ``right special.''
	    If (III) holds, then $K_\ell = K_r = K-1$ and necessarily $\#\PRECOLOR(\nu) = 2$ for all $\nu$.
	    A similar argument to the above works here as well.
	    Namely, there must be at least $K-3$ colored binary loops,
		    and if we wait for one of them to change, it cannot perform an SBS move.
	    If it performs an RBS or WBS move,
		    then the right special vertex in the loop must share its color with the vertex $a$ on the path away from the left special vertex.
	    But since $\#\PRECOLOR(\nu)=2$ for all $\nu$, we immediately contradict minimality.
      \end{proof}

	\section{Further Work}\label{SecFuture}

	For large $K$, the statement ``$\#\ERGODICsp = K-2$'' for $\SPACE$ satisfying \eqref{EqFixedKComp}
	    already seems problematic.
	We plan to expand the results presented here to explore improvements to Theorem \ref{ThmMain}.

	As discussed in the introduction,
		an interesting class of shifts satisfying \eqref{EqFixedKComp}
		are generated by interval exchange transformations.
	By \cite[Lemma 8]{cFerenZam2008} these shifts satisfy the binary extension condition.
	Furthermore, they enjoy a \emph{regular bispecial condition}, meaning
	    the bispecial words of length $n$ are regular bispecial for all large $n$.
	Using the additional assumption that $\SPACE$ satisfies the regular bispecial condition,
	      we have already achieved a bound $\#\ERGODICsp \leq C \cdot K$ for a constant $C< 1$.
	These results will be produced in a future paper.
	
	We aim to sharpen the bounds for shifts with either the binary extension condition
	      or regular bispecial condition and compare these bounds with those for interval
	      exchange transformations, $\#\ERGODICsp\leq(K+1)/2$.
      \bibliographystyle{abbrv}
      \bibliography{dfbibfile}
\end{document}